\documentclass[11pt]{amsart}
\usepackage[notref,notcite]{showkeys}

\usepackage{theoremref}
\usepackage{mathrsfs}
\usepackage{hyperref}

\usepackage{tikz}
\usepackage{times}
\usepackage{color}
\usepackage{amsmath}
\usepackage{amssymb}
\usepackage{amsthm}
\usepackage{enumerate}
\usepackage{amsbsy}
\usepackage{amsfonts}

\usepackage[T1]{fontenc}
\usepackage[bitstream-charter]{mathdesign}
\usepackage{eucal}

\newtheorem{theo}{Theorem}[section]

\newtheorem{prop}[theo]{Proposition}

\newtheorem{lemm}[theo]{Lemma}
\newtheorem{rem}[theo]{Remark}

\newcommand{\al}{\alpha}
\newcommand{\be}{\beta}
\newcommand{\ga}{\gamma}
\newcommand{\Ga}{\Gamma}
\newcommand{\la}{\lambda}

\newcommand{\ep}{\epsilon }
\newcommand{\te}{\theta}
\newcommand{\De}{\Delta}
\newcommand{\de}{\delta}

\newcommand{\pa}{\partial}

\newcommand{\R}{{\mathbb R}^3}

\newcommand{\ri}{\rightarrow}

\newcommand{\na}{\nabla}

\newcommand{\N}{\mathbb{N}}
\newcommand{\bR}{{\mathbb R}}

\DeclareMathOperator*{\divg}{div}

\begin{document}
\baselineskip=18pt

\title[]{Global existence of solutions for the Hall-MHD equations}

\author{TongKeun Chang, Bataa Lkhagvasuren and MinSuk Yang}

\address{TongKeun chang: Department of
Mathematics, Yonsei University, 50 Yonsei-ro, Seodaemun-gu, Seoul,
South Korea 120-749 } 
\email{chang7357@yonsei.ac.kr }

\address{MinSuk Yang:  Department of
Mathematics, Yonsei University, 50 Yonsei-ro, Seodaemun-gu, Seoul,
South Korea 120-749 } 
\email{m.yang@yonsei.ac.kr  }

\address{Bataa Lkhagvasuren:  Mathematics Department, Chonnam National University,
Gwangju, South Korea}
\email{bataa@chonnam.ac.kr}

\thanks{T. Chang was supported by NRF-2020R1A2C1A01102531, B. Lkhagvasuren was supported by NRF-2022R1I1A1A01055459, and M. Yang was supported by NRF-2021R1A2C4002840.
}

\begin{abstract}
In this paper, we study the global well-posedness of  the incompressible Hall-MHD system for small initial data $( u_0, b_0) \in \dot B^{\frac3p -1}_{p,5} (\R) \times \Big( \dot B^{\frac3p -1}_{p,5} (\R) \cap \dot B^{\frac3p}_{p,1}(\R) \Big)$ for $1 < p < 5$.   
We get the result under the weaker regularity and integrability conditions of the initial data than the previous works.
We also give integral formulae for the solution. \\

\end{abstract}

\maketitle

\section{\bf Introduction}
\label{S1}
\setcounter{equation}{0}
In this paper, we study the Cauchy problem of the incompressible Hall-MHD system
\begin{equation}
\label{E11}
\begin{cases} \vspace{2mm}
u_t -\De  u + (u \cdot \na) u - ( \na \times b) \times b  +\na p = 0 & \quad \R \times (0, \infty),\\
\vspace{2mm} b_t - \De  b +\na  \times ((\na \times b) \times b) -\na \times ( u \times b)   =0  & \quad \R \times (0, \infty),\\
\vspace{2mm}
\divg u = \divg b =0,\\
u|_{t =0} = u_0, \,\, b|_{t =0} = b_0,
\end{cases}
\end{equation}
where $u$ and $b$ are the velocity and the magnetic field, and $p$ is the pressure.
The initial velocity field $u_0$ and the initial magnetic field $b_0$ satisfy $\divg u_0 = \divg b_0 = 0$. 
The Hall term $\na \times ((\na \times b) \times b)$ makes the Hall-MHD system significantly different in the mathematical theory from the MHD system.
It also prevents straightforward adaptations from arguments used in the mathematical analysis of MHD and related models.

We study the global existence of the solution of \eqref{E11} in $C_b ([0, T]; \dot B^{\al}_{p,q}(\R))$.  
We list here a few known results related to this problem. 
Chae et al. \cite{CDG} established the global existence of weak solutions as well as the local well-posedness   for initial data in $H^{s} (\R)$ with $ s> \frac52$.
In \cite{CL}, authors showed the global well-posedness of smooth solutions for small initial data norms $ \| u_0\|_{\dot H^{\frac32} } +\| b_0 \|_{\dot H^{\frac32}}$ or $ \| u_0\|_{\dot B^{\frac12}_{2,1} } +\| b_0 \|_{\dot B^{\frac32}_{2,1}}$ (also see \cite{BF}).  
Wu, Yu  and Tang \cite{WYY } showed the local existence of solution when initial data in $H^s (\R)$ for $\frac32 < s < \frac52$ and global existence of solution when $\| u_0\|_{H^s({\mathbb R}^3) } $ and $\| b_0\|_{H^s({\mathbb R}^3) } $ are sufficiently  small. Dai  \cite{Dai} showed the local existence of solution for inital data  in  $ u_0 \in   H^{s-1}({\mathbb R}^{3})$ and $ b_0 \in H^{s +\ep}({\mathbb R}^{3}) , s > \frac32, \,\, \ep > 0$  (also see \cite{Zs}).  Danchin and Tan \cite{DT} showed the global existence of solution when the initial data is in  $ u_0 \in   \dot B^{\frac{3}p-1}_{p,1}({\mathbb R}^{3})$ and $ b_0 \in \dot B^{\frac{3}p -1}_{p,1}({\mathbb R}^{3}) \cap \dot B^{\frac{3}p}_{p,1}({\mathbb R}^{3}),  1 \leq  p\leq  \infty$ with small initial data  (see  \cite{WZ}).  Zhang \cite{Zn} showed the global existence of solution when the initial data in  $u_0, \, b_0 \in  H^s (\R)$ with $s\ge 3$ have sufficiently small norm $ \|u_0\|_{\dot B^{\frac12}_{2,\infty}} +  \|b_0\|_{\dot B^{\frac12}_{2,\infty}} +  \|b_0\|_{\dot B^{\frac32}_{2,\infty}}$.

We consider the following integral formulas to \eqref{E11}: 
\begin{align}
\label{E12}
\begin{split}
u(x,t)  & = \int_{\R} \Ga(x-y, t) u_0 (y) dy +  \int_0^t \int_{\R} \na \Ga (x-y, t-s) {\mathbb P} (u \otimes u ) (y,s)dyds\\
& \qquad  - \int_0^t \int_{\R} \na \Ga(x-y, t-s)   :  {\mathbb P}\big(  ( b \otimes b) \big)(y,s) dyds  +  \frac12  \int_0^t \int_{\R}    {\mathbb P} \na \Ga(x-y, t-s)  | b (y,s)|^2 dyds,
\end{split}\\
\label{E13}
\begin{split}
b(x,t) &  = \int_{\R} \Ga(x -y, t) b_0 (y) dy +  \int_0^t \int_{\R}   \na^2 \Ga(x-y,t-s)b(y,s) \times b(y,s)   dyds\\
& \qquad  +  \int_0^t \int_{\R } \na \Ga(x-y, t-s) \times ( u\times b)(y,s) dyds,
\end{split}
\end{align} 
where  $ \Ga $ is the fundamental solution of the heat equation in $\R$ and ${\mathbb P}$ is the Helhmotz projection operator in $\R$.  In the integral formula \eqref{E12} and \eqref{E13},  the most difficult term to manage is 
\begin{align}
\label{E14}
T (b,b) (x,t) = 
\int_0^t \int_{\R}   \na^2 \Ga(x-y,t-s)b(y,s) \times b(y,s)   dyds
\end{align}
induced from Hall term.  We show that $ T$ is bounded on the homogeneous anisotropic Besov space $\dot B^{s, \frac{s}2}_{p,q} (\R \times (0, \infty))$ (see  Section \ref{S2}  for the definition). 
From it, we can estimate the solution $ (u, b)$ with the homogeneous anisotropic Besov space norm and find integral formulas \eqref{E12} and \eqref{E13} for the solution $(u,b)$.

Here are our main results.

\begin{theo}
\label{T1}
Let $u_0 \in \dot B^{\frac{3}p -1}_{p, 5} (\R)  \cap \dot B^{\al -1 -\frac2p}_{p, q} (\R)$ and $b_0 \in  \dot B^{\frac{3}p-1}_{p, 5} (\R) \,  \cap \, \dot B^{\frac{3}p }_{p,1} (\R)   \cap \dot B^{\al -\frac2p}_{p, q} (\R)$ with $\divg  u_0 = \divg  \, b_0 =0$ for some $p, q, \alpha$ satisfying $1< p < 5$, $1 \leq q \leq \infty$, and $\frac{5}p < \al < \infty$.
Then there is a positive constant $\ep_0$ depending only on $p$ and $\al$ such that if
\begin{align}
\| u_0\|_{\dot B^{  \frac{3}p -1}_{p, 5} (\R)}  + \| b_0\|_{\dot B^{  \frac{3}p -1 }_{p, 5} (\R)} + \| b_0\|_{\dot B^{  \frac{3}p  }_{p, 1} (\R)} < \ep_0,
\end{align}
then the equations \eqref{E11} have a solution $(u,b)$ satisfying 
\begin{align*}
u &\in   \dot B^{\al-1, \frac{\al -1}2}_{p,q} (\R \times (0, \infty)) \cap  \dot B^{\frac{5}p-1, \frac5{2p} -\frac12 }_{p,5} (\R \times (0, \infty)),\\
b &\in   \dot B^{\al ,\frac{\al}2 }_{p,q} (\R \times (0, \infty)) \cap  \dot B^{\frac{5}p, \frac5{2p} }_{p,1}  (\R \times (0, \infty)) \cap \dot B^{\frac{5}p -1,\frac5{2p} -\frac12  }_{p,5}  (\R \times (0, \infty)).
\end{align*}
Moreover, the integral formulas \eqref{E12} and \eqref{E13} hold.
\end{theo}

To prove Theroem \ref{T1}, we shall use the boundedness of the operator  $T $ defined in \eqref{E14} on $\dot B^{s, \frac{s}2}_{p,q}(\R \times (0, \infty))$ for $ 1 < p < \infty, \,\, 1 \leq  q \leq \infty$ and $0 < s < \infty$. 
More precisely,  the parabolic type Caldron-Zygmund singular integral theorem implies that $T$ is bounded on $\dot W^{2k, k}_{p} (\R \times (0, \infty))$ for $ k \in \N_0 $ and $ 1 < p < \infty$ (see Section \ref{S2} for the definition of $\dot W^{2k, k}_{p} (\R  \times (0, \infty))$). Since $ \dot B^{s, \frac{s}2}_{p,q} (\R \times (0, \infty)),  \,\, 1 \leq p, \, q \leq \infty$, $0 < s < \infty$ is an interpolation space  of $\dot W^{2k, k}_{p} (\R  \times (0, \infty))$ (see (1) of Proposition \ref{P24}),   $ T$ defined in \eqref{E14} is also bounded on $ \dot B^{s, \frac{s}2}_{p,q} (\R \times (0, \infty))$, $0 < s < \infty$ (see (2) of Proposition \ref{P32}). 
From Lemma \ref{L25} and Lemma \ref{L26}, we get the following estimate
 \begin{align*}
 \| T( b, b)\|_{\dot B^{s, \frac{s}2}_{p,q} (\R \times (0, \infty))} 
&\lesssim \sum_{1 \leq i, j \leq 3}\|  b_ib_j\|_{\dot B^{s, \frac{s}2}_{p,q} (\R \times (0, \infty))} \\
&\lesssim \|  b\|_{L^\infty (\R \times (0, \infty))}\|   b\|_{\dot B^{s, \frac{s}2}_{p,q} (\R \times (0, \infty))}\\
&\lesssim \|  b\|_{\dot B^{\frac{5}p, \frac{5}{2p}}_{p, 1} (\R \times (0, \infty))}\|   b\|_{\dot B^{s, \frac{s}2}_{p,q} (\R \times (0, \infty))}.
 \end{align*}
Using this inequality and an iteration method, we can show the existence of solutions to \eqref{E11}.

Using Lemma \ref{L28}, we can obtain the following theorem.

\begin{theo}
\label{T2}
Let  $ 1 < p < 3$, $ 1 +\frac2p< \al$ and  $(u, b)$ be the solution found in Theorem \ref{T1}. Then
\begin{align*}
u &\in  C_b ([0, \infty); \dot B^{\al-1-\frac2p}_{p,q} (\R)) \cap  C_b ([0, \infty); \dot B^{\frac{3}p-1 }_{p,5} (\R)),\\
b &\in  C_b ([0, \infty); \dot B^{\al  -\frac2p}_{p,q} (\R)) \cap C_b ([0, \infty); \dot B^{\frac{3}p }_{p,1}  (\R)) \cap C_b ([0, \infty); \dot B^{\frac{3}p -1 }_{p,5}  (\R)).
\end{align*}
\end{theo}

\begin{rem}
\begin{itemize}

\item[(1)]
Compared to the results in   \cite{CDG},  \cite{CL}, \cite{BF},  \cite{WYY }, \cite{Dai} and  \cite{Zs},  we   improve the range of integrability $ 1 < p < 3$ and regularity $1 +\frac2p < \al < \infty$.

\item[(2)]
Compared to the results in \cite{DT}, we extend the range of $ q = 5$ and $1 +  \frac2p < \al < \infty$.

\item[(3)]
As the author commented in \cite{Zn}, the result in it is restricted to $p =2$.

\item[(4)]
Our solution is represented by the integral formula with \eqref{E12} and \eqref{E13}.

\end{itemize}
\end{rem}

This paper is organized as follows. 
In Section \ref{S2}, we recall some known results and define homogeneous anisotropic Besov space. 
We also introduce basic properties of homogeneous anisotropic Besov space and estimates that are useful for our purpose. 
%Section 3 is devoted to constructing a very weak solution
%in a half-space such that the norm of its gradient is not bounded. 
In Section \ref{S3}, we present the proof of Theorem 1.1.

\section{Notations and Definitions}
\label{S2}
\setcounter{equation}{0}

Throughout this paper, we denote a generic constant by $c$, which can be changed from line to line. 
We denote $D^{k}_x D^{m}_t = \frac{\pa^{|k|}}{\pa x^{k}} \frac{\pa^{m} }{\pa t}$ for multi-indices $k$ and $m \in \N_0: = {\mathbb N} \cup \{0\}$.
We denote by $X'$ the dual space of a Banach space $X$ and by $L^p(0, \infty;X)$ the usual Bochner space for $1\leq p\leq \infty$.
For $0< \theta<1$ and $1\leq p\leq \infty$, we denote by $(X,Y)_{\theta,p}$ and $[X, Y]_\te$  the real interpolation spaces and the complex interpolation spaces between the Banach spaces $X$ and $Y$, respectively.
For $1\leq p\leq \infty$, we write the conjugate exponent of $p$ by $p'$, i.e., $1/p+1/p'=1$.

For $k \in \N_0$, $\dot W^k_p ({\mathbb R}^3), \,  1\leq p\leq \infty$ is  the usual homogeneous Sobolev space in ${\mathbb R}^3$. 
Analogously, for $s\in {\mathbb R}$, $\dot B^s_{p,q} ({\mathbb R}^3), 1\leq p,q\leq \infty$  is the usual homogeneous Besov space in ${\mathbb R}^3$.

Let $s\in {\mathbb R}$ and $1\leq p \leq \infty$. We define an
anisotropic homogeneous Sobolev space $\dot W^{s, \frac{s}2}_p ({\mathbb R}^{n+1})$, $ n \geq 1$ by
\begin{eqnarray*}
\dot W^{s, \frac{s}2}_p  ({\mathbb R}^{n+1}) = \{ f \in {\mathcal S}'({\mathbb
R}^{n+1}) \, | \, f  = h_{s} * g, \quad \mbox{for some} \quad g   \in L^p ({\mathbb R}^{n+1}) \}
\end{eqnarray*}
with norm $\|f\|_{\dot W^{s, \frac{s}2}_p  ({\mathbb R}^{n+1})} : = \| g \|_{L^p({\mathbb R}^{n+1})} \,  = \|  h_{-s} * f  \|_{L^p({\mathbb R}^{n+1})},$
 where
$*$ is a convolution in ${\mathbb R}^{n+1}$ and ${\mathcal S}^{'}({\mathbb
R}^{n+1})$
 is the dual space of the Schwartz space
${\mathcal S}({\mathbb R}^{n+1})$.
Here $h_s $  is a distribution in ${\mathbb R}^{n+1}$ whose Fourier transform
in ${\mathbb R}^{n+1}$ is defined by
\begin{eqnarray*}
{\mathcal F}_{x,t} ( h_{s}) (\xi,\tau) = c_s(  4\pi^2 |\xi|^2 + i
\tau)^{-\frac{s}{2}}, \quad (\xi, \tau) \in {\mathbb R}^n \times {\mathbb R},
\end{eqnarray*}
where ${\mathcal F}_{x,t}$ is the Fourier transform in ${\mathbb R}^{n+1}$. 
Here the function $(  4\pi^2 |\xi|^2 + i
\tau)^{-\frac{s}{2}}$ is defined  as  $e^{-\frac{s}2 \ln(4\pi^2 |\xi|^2 + i\tau)}$, where the branch line of the logarithm that  is a negative  real line for real arguments is chosen.

Using the multiplier theorem for $L^p({\mathbb R}^{n+1})$ (see  Section 2 in \cite{CK} or Section 6 in \cite{BL}), in  the case  $\alpha=k \in {\mathbb N} \cup \{ 0\}$,  
\begin{align*}
 \|f\|_{\dot W^{2k, k}_p  ({\mathbb R}^{n+1})  } & \approx \sum_{|\be| +2 l = 2k} \| D_{x^\be}^\be D^{l}_{t}   f\|_{L^p ({\mathbb R}^{n+1})}.
\end{align*}
%where $\displaystyle D^\frac12_t f(t) = D_t \int_{-\infty}^t \frac{f(s)}{(t -s)^\frac12} ds$ and $D^{k  +\frac12}_tf = D^\frac12_t D^k_tf$.
In particular,  $\displaystyle\| f\|_{\dot W^{0,0}_p({\mathbb R}^{n+1})}  = \| f\|_{ L^p({\mathbb R}^{n+1})}$ and 
\begin{align*}
\| f\|_{\dot W^{2, 1 }_{p}({\mathbb R}^{n+1}) }
            =  \| D_t  f\|_{L^p({\mathbb R}^{n+1})}
               + \|D^2_x f\|_{L^p({\mathbb R}^{n+1})}.
\end{align*}
See  \cite{CK}.

Fix a Schwartz function $\phi\in {\mathcal S} ({\mathbb R}^{4})$ satisfying $\widehat{\phi}(\xi,\tau) > 0$ for $\frac12 < |\xi| + |\tau|^\frac12 < 2$, $\widehat{\phi}(\xi,\tau)=0$ elsewhere, and
$\sum_{j \in \mathbb{Z}} \widehat{\phi}(2^{-j}\xi, 2^{-2j}\tau) =1$ for $(\xi,\tau) \neq 0$.
Let
\begin{align*}
\widehat{\phi_j}(\xi,\tau) &:= \widehat{\phi}(2^{-j} \xi, 2^{-2j}\tau) \quad (j = 0, \pm 1, \pm 2 , \cdots) \\
\widehat{\psi}(\xi,\tau) &:= 1- \sum_{j=1}^\infty \widehat{\phi} (2^{-j} \xi, 2^{-2j} \tau).
\end{align*}
We define the homogeneous anisotropic Besov spaces 
\begin{align*}
\dot B^{s, \frac{s}2}_{p,q} ({\mathbb R}^{4}) = \{ f \in {\mathcal S}'({\mathbb R}^{4}) \mid \| f\|_{\dot B^{s, \frac{s}2}_{p,q} ({\mathbb R}^{4})  } < \infty \}
\end{align*}
with the norm
\begin{align*}
\| f\|_{\dot B^{s, \frac{s}2}_{p,q} ({\mathbb R}^{4})} & =  \left(\sum_{j\in \mathbb{Z}}  2^{q s j}\| f* \phi_j \|^q_{L^p ({\mathbb R}^{4})} \right)^\frac1q,
\end{align*}
where $*$ is $4$ dimensional convolution.

The properties of the anisotropic Besov spaces are similar to those of usual Besov spaces. In particular, the following properties will be used later.
\begin{prop}
\label{P21}
Let $s_0, s_1 \in \bR$.
\begin{itemize}
\item[(1)]
For  $1 \leq  p, \,q \leq  \infty$,
\begin{align*}
&( \dot W^{ s_0, \frac{s_0}2}_p   ({\mathbb R}^{4}),   \dot W^{ s_1, \frac{s_1}2}_p ({\mathbb R}^{4} ))_{\te, q} = \dot B^{s, \frac{s}2}_{p,q} ({\mathbb R}^{4}), \quad [ \dot W^{ s_0, \frac{s_0}2}_p   ({\mathbb R}^{4}),   \dot W^{ s_1, \frac{s_1}2}_p ({\mathbb R}^{4} )]_{\te} = \dot W^{s, \frac{s}2}_{p} ({\mathbb R}^{4}) % \\
%&  (\dot W^{s, \frac{s}2 }_{p_0}   ({\mathbb R}^{4}),   \dot W^{s, \frac{s}2 }_{p_1} ({\mathbb R}^{4} ))_{\te, r} = \dot B^{s, \frac{s}2}_{qr} ({\mathbb R}^{4}),
\end{align*}
where $s = s_0  (1 -\te) + s_1 \te$ and $0 < \te<1$, . % and $\frac1q = \frac{1-\te}{p_0} + \frac{\te}{p_1}$.

\item[(2)]
If $ 1 \leq p_0\leq p_1  \leq \infty, \,  1 \leq r_0\leq r_1  \leq \infty$ and $ s_0\geq s_1$ with $ s_0 - \frac{5}{p_0} =s_1 - \frac{5}{p_1}$, then 
\begin{align*}
\dot  W^{ s_0,  \frac{s_0}2  }_{p_0} ({\mathbb R}^{4} ) \subset   \dot W^{s_1, \frac{s_1}2  }_{p_1} ({\mathbb R}^{4}), \quad
\dot  B^{ s_0, \frac{ s_0}2  }_{p_0, r_0} ({\mathbb R}^{4} ) \subset   \dot B^{s_1,  \frac{s_1}2  }_{p_1, r_1} ({\mathbb R}^{4}).
\end{align*}

\item[(3)]
For $0 < s < \infty$ and $ 1  < p  <  \infty$,
\begin{align*}
\dot W^{s,\frac{s}2}_{p} ({\mathbb R}^{4})   &=L^p ({\mathbb R}; \dot W^{s}_{p} (\R)) \cap L^p (\R; \dot W^{\frac{s}2}_{p}({\mathbb R})).
 \end{align*}

\item[(4)]
If   $f \in   \dot B_{p,q}^{s, \frac{s}2} ({\mathbb R}^{4})$ or $f \in   \dot W_{p}^{s, \frac{s}2} ({\mathbb R}^{4}), \, \frac{2}p < s < \infty$, then $f|_{t =0 } \in  \dot  B_{p,q}^{s -\frac{2}p }(\R )$ or $f|_{t =0 } \in  \dot  B_{p,p}^{s -\frac{2}p }(\R )$ with
\begin{align*}
\| f|_{t=0} \|_{ \dot B_{p,p}^{s -\frac{2}p } (\R )} \lesssim \| f \|_{\dot  W_{p}^{s,\frac{s}2} ({\mathbb R}^{4})}, \quad
\| f|_{t=0} \|_{ \dot B_{p,q}^{s -\frac{2}p } (\R )} \lesssim \| f \|_{\dot  B_{p,q}^{s,\frac{s}2} ({\mathbb R}^{4})}.
\end{align*}

%\item[(5)]
%\begin{align}
%\| D_x f\|_{\dot B^{s, \frac{s}2}_{pr}({\mathbb R}^{4})} \leq c \| f\|_{\dot B^{\al (s -\frac1\al), s -\frac1\al}_{pr} ({\mathbb R}^{4})}.
%\end{align}

\end{itemize}
\end{prop}

\begin{rem}
We refer to Theorem 3.7.1 in \cite{amman-anisotropic} and Theorem 6.4.5 in \cite{BL} for the proof of (1), and Theorem 3.7.3 in \cite{amman-anisotropic}, Theorem 6.5.1 in \cite{BL} for (2) and Theorem 3.6.7 in \cite{amman-anisotropic} for (3) and Theorem 4.5.2 in \cite{amman-anisotropic} and  Theorem 6.6.1 in \cite{BL} for (4).
\end{rem}

The space $\dot  {W}^{s, \frac s2 }_{p} (\R \times (0, \infty))$ is defined to be the restriction of $ \dot {W}^{s, \frac s2 }_{p} ({\mathbb R}^4)$ with the norm
\begin{align}\label{1004-1}
\|f\|_{ \dot W^{s, \frac{s}2 }_{p} (\R \times (0, \infty))}=\inf\{ \|F\|_{ \dot W^{s, \frac{s}2 }_{p}({\mathbb R}^4)} \mid F\in  W^{s, \frac{s}2}_{p}({\mathbb R}^4)\mbox{ with } F|_{\R \times (0, \infty)}=f\}.
\end{align}
Similarly, the space $\dot  {B}^{s, \frac s2 }_{p,q} (\R \times (0, \infty))$ is defined to be the restriction of $\dot  {B}^{s, \frac s2 }_{p,q} ({\mathbb R}^4)$ in the same sense.
To explain interpolation spaces defined in $\R \times (0, \infty)$, we introduce retraction and coretraction. Let $A$ and $B$ be two Banach spaces. An operator $R: A \ri B$ is said to be retraction if there is an operator $S: B \ri A$ such that $RS = I_B$, where  $I_B: B \ri B$ is the identity operator in $B$. 
The operator $S$ is said to be a coretraction with respect to $R$. 
The following proposition is the conclusion of Theorem 1.2.4 in \cite{Tr}.

\begin{prop}
\label{P23}
Let $A_0, \,  A_1, \, B_0$ and $B_1$ be Banach spaces. 
If $R: A_i \ri B_i, \, i = 0,1$ is a retraction and $S: B_i \ri A_i, \, i = 0,1$ its coretraction, then 
\begin{align*}
&S: [B_0, B_1]_\te \ri  \{ SR a \mid a \in [A_0, A_1]_\te \} = \{ S b \mid b  \in R([A_0, A_1]_\te) \},\\
&S: (B_0, B_1)_{\te,q} \ri \{  SR a \mid a \in (A_0, A_1)_{\te,q} \}   = \{ S b \mid b \in R((A_0, A_1)_\te) \}
\end{align*}
are bijections.
\end{prop}

\begin{rem}\label{rem0920}
From the definition of retraction $R:A_i \ri B_i $, coretraction  and the properties of interpolation spaces, $ R: (A_0, A_1)_{\te, q} \ri (B_0, B_1)_{\te,q}$ is surjective and $ S: (B_0, B_1)_{\te, q} \ri (A_0, A_1)_{\te,q}$ is injective.    Hence, for $ b \in  (B_0, B_1)_{\te, q}$, there is $ a \in (A_0, A_1)_{\te, q}$ such that $R a = b$. Then, from Proposition \ref{P23}, we have 
\begin{align*}
\| b \|_{(B_0, B_1)_{\te, q}} = \| Ra \|_{(B_0, B_1)_{\te, q}} \approx \|  SRa \|_{(A_0, A_1)_{\te, q}} =\|  Sb \|_{(A_0, A_1)_{\te, q}}.
\end{align*}
Similarly, we have 
\begin{align*}
\| b \|_{[B_0, B_1]_{\te}} = \| Ra \|_{[B_0, B_1]_{\te}} \approx \|  SRa \|_{[A_0, A_1]_{\te}} = \|  Sb \|_{[A_0, A_1]_{\te}}.
\end{align*}
\end{rem}

For $k \in \N_0$, we define $\dot {\mathbb W}_p^{2 k, k} (\R \times (0,\infty))$ to be the space of all $f \in {\mathcal S}'({\mathbb R}^{4})$ such that 
\[\| f\|_{\dot {\mathbb W}^{2k, k}_p (\R \times (0, \infty))} : =\sum_{2l + |\be|  =2 k} \| D_t^l D_x^{\be} f\|_{L^p (\R \times (0, \infty))} < \infty,\]
where $\be = (\be_1,  \cdots, \be_n) \in \N_0^n $ and $|\be | = \sum_{1 \leq i \leq n} \be_i $. We also define the interpolation spaces
\begin{align*}
\dot {\mathbb W}^{s, \frac{s}2}_{p} (\R \times (0, \infty))& : =  [L^{p}   (\R \times (0, \infty)),   \dot {\mathbb W}^{2k, k}_{p} (\R \times (0, \infty))]_{\te},\\
%\dot {\mathbb B}^{s, \frac{s}2}_{qr} (\R \times (0, \infty))& : =  (L^{p_0}   (\R \times (0, \infty)),   \dot {\mathbb W}^{s, \frac{s}2}_{p_1} (\R \times (0, \infty)))_{\te,r},\\
\dot {\mathbb B}^{s, \frac{s}2}_{p,q} (\R \times (0, \infty))& :  = ( L^p   (\R \times (0, \infty)),   \dot {\mathbb W}^{2 k, k}_p (\R \times (0, \infty)))_{\te, q},
\end{align*}
where $0 < \te<1$ and  $s = 2 k \te$.% and $\frac1q = %\frac{1-\te}{p_0} + \frac{\te}{p_1}$.

Let $k \in \N_0$.
For $f \in \dot {\mathbb W}^{2 k, k}_p (\R \times (0, \infty))$, we define   the extension $E f \in \dot W^{2 k, k}_p ({\mathbb R}^{4})$  by
$E f (x,t) =
 f(x,t)$ for $t > 0$ and $E f(x,t) = \sum_{j =1}^{k+1} \la_j f(x,-jt)$ for $t < 0$, where $\la_1, \, \la_2, \cdots, \la_{k+1}$ satisfy $\sum_{j =1}^{k+1} (-j)^l \la_j =1, \,\,  0 \leq l \leq k$ (see Theorem 5.19 in \cite{AF}).  
Then $E f \in \dot W^{2 k, k}_p ({\mathbb R}^{4})$ with 
 \[\| E f\|_{\dot W^{2 k, k}_p ({\mathbb R}^{4})} \leq c \| f \|_{\dot {\mathbb W}^{2k, k}_p (\R \times (0, \infty))}.\]
 This implies that $ \dot {\mathbb W}^{2 k, k}_p (\R \times (0, \infty)) =  \dot W^{2 k, k}_p (\R \times (0, \infty))$ for $k \in \N_0 $.

Let $s > 0$ and $s \not\in \N$.
If $ f \in \dot {\mathbb W}^{s, \frac{s}2}_p (\R \times (0, \infty))$, then $ E f \in \dot W^{s , \frac{s}2}_p ({\mathbb R}^{4})$ by the property of the complex interpolation (see (1) of  Proposition \ref{P21}), and so $f = Ef|_{\R \times (0, \infty)} \in  \dot W^{s, \frac{s}2 }_p (\R \times (0, \infty))$. 
Hence 
\begin{align}
\label{E23}
 \dot {\mathbb W}^{s, \frac{s}2}_p (\R \times (0, \infty)) \subset \dot W^{s, \frac{s}2 }_p (\R \times (0, \infty)).
 \end{align}
%Similarly, we have
 %\begin{align}\label{E24}
 %\dot {\mathbb B}^{s, \frac{s}2}_{p,q} (\R \times (0, \infty)) \subset \dot B^{s, \frac{s}2}_{p,q} (\R \times (0, \infty)).
 %\end{align}

Let $2k < s < 2k+2$ for $k \in \N_0$.
Then,  $E:  W^{2 i,i}_p (\R \times (0, \infty)) \ri   W^{2 i,i}_p ( {\mathbb R}^{4}), \,\, i = k, \, k+1$   are retraction with respect to coretraction $ R : W^{2i,i}_p ({\mathbb R}^{4}) \ri W^{2 i,i}_p (\R \times   (0, \infty)), \,\, RF = F|_{\R \times (0, \infty)}$.
Let $f \in \dot W^{s, \frac{s}2 }_p(\R \times (0, \infty))$. 
From the definition of $ \dot W^{s, \frac{s}2 }_p (\R \times (0, \infty))$ (see \eqref{1004-1}), there is $F \in \dot W^{s, \frac{s}2 }_p ({\mathbb R}^{4})$ such that $R F =f$ and $\| F \|_{\dot W^{s, \frac{s}2 }_p ({\mathbb R}^{4})} \lesssim \| f\|_{\dot W^{s, \frac{s}2 }_p (\R \times (0, \infty))} $.

\begin{align*}
\| b \|_{[B_0, B_1]_{\te}} = \| Ra \|_{[B_0, B_1]_{\te}} \approx \|  SRa \|_{[A_0, A_1]_{\te}} \leq c \|  Sb \|_{[A_0, A_1]_{\te}}.
\end{align*}

 From  Remark \ref{rem0920}, 
 \begin{align*}
%\label{E24}
\|f \|_{\dot {\mathbb W}^{s, \frac{s}2}_p (\R \times (0, \infty))}  = \|R F \|_{\dot {\mathbb W}^{s, \frac{s}2}_p (\R \times (0, \infty))} \approx \| SRF \|_{W^{s, \frac{s}2}_p ({\mathbb R}^4 )} \leq c \| F \|_{W^{s, \frac{s}2}_p ({\mathbb R}^4 )} \leq c \|f \|_{\dot W^{s, \frac{s}2}_p (\R \times (0, \infty))}.
\end{align*}
 This implies
\begin{align}
\label{E25}
\dot W^{s, \frac{s}2 }_p (\R \times (0, \infty)) \subset \dot {\mathbb W}^{s, \frac{s}2}_p (\R \times (0, \infty)).
\end{align}
From \eqref{E23} and \eqref{E25}, we have for $1 \le p \le \infty$ and $0 < s < \infty$,
\begin{align}
\label{E26}
\dot {\mathbb W}^{s, \frac{s}2}_p (\R \times (0, \infty)) =  \dot W^{s, \frac{s}2 }_p (\R \times (0, \infty)).
\end{align}
Similarly, we obtain that for $1 \le p, q \le \infty$ and $0 < s < \infty$,
\begin{align}
\label{E27}
\dot {\mathbb B}^{s, \frac{s}2}_{p,q} (\R \times (0, \infty)) =  \dot B^{s, \frac{s}2 }_{p,q} (\R \times (0, \infty)).
\end{align}
From now on, we use the notations $ L^p = L^p(\R \times (0, \infty))$,  $ \dot W^{s, \frac{s}2 }_{p} = \dot W^{s, \frac{s}2 }_{p} (\R \times (0, \infty)) $ and   $\dot B^{s,\frac{ s}2}_{p,q} =  \dot B^{s, \frac{s}2}_{p,q} (\R \times (0, \infty))$.
Note that
\begin{align*}
\| E f\|_{\dot W^{s, \frac{s}2 }_{p} ({\mathbb R}^{4})} \approx \| f\|_{\dot W^{s, \frac{s}2 }_{p}}, \quad
\| E f\|_{\dot B^{s, \frac{s}2}_{pr} ({\mathbb R}^{4})} \approx \| f\|_{\dot B^{s, \frac{s}2}_{pr}}.
\end{align*}

From Proposition \ref{P21}, we obtain the following results.
\begin{prop}
\label{P24} 
\begin{itemize}
\item[(1)]
For  $1 \leq  p,\, q \leq  \infty$,
\begin{align*}
( \dot W^{s_0, \frac{s_0}2}_p,   \dot W^{s_1, \frac{s_1}2}_p )_{\te, q} = \dot B^{s, \frac{s}2}_{p,q},  
 % (\dot W^{s, \frac{s}2 }_{p_0},   \dot W^{s, \frac{s}2 }_{p_1})_{\te, r} = \dot B^{s, \frac{s}2}_{qr},
\end{align*}
where $0 < \te<1$ and  $s = s_0  (1 -\te) + s_1 \te$. % and $\frac1q = \frac{1-\te}{p_0} + \frac{\te}{p_1}$.

\item[(2)]
Let  $ 1 \leq p_0\leq p_1  \leq \infty, \,  1 \leq q_0\leq q_1  \leq \infty$ and $ s_0\geq s_1$ with $s_0 - \frac{5}{p_0} = s_1 - \frac{5}{p_1}$.
Then, the following inclusions hold
\begin{align*}
\dot  W^{s_0,  \frac{s_0}2  }_{p_0}  \subset   \dot W^{s_1, \frac{s_1}2  }_{p_1}, \quad
\dot  B^{ s_0, \frac{ s_0}2  }_{p_0, q_0} \subset   \dot B^{s_1, \frac{ s_1}2  }_{p_1 ,q_1} .
\end{align*}

\item[(3)]
For $0 < s$ and $ 1  < p  <  \infty$,
\begin{align*}
\dot W^{s,\frac{s}2}_{p}   &=L^p (0,\infty; \dot W^{ s}_{p} (\R)) \cap L^p (\R; \dot W^{\frac{s}2}_{p}(0, \infty)).
 \end{align*}

 \item[(4)]
If   $f \in   \dot B_{p,q}^{s, \frac{s}2} $ or $f \in   \dot W_{p}^{s, \frac{s}2}, \, s > \frac{2}p$, then $f|_{t =0 } \in  \dot  B_{p,q}^{s -\frac{2}p }(\R )$ with
\begin{align*}
\| f|_{t=0} \|_{ \dot B_{p,p}^{ s -\frac{2}p } (\R )} \leq c \| f \|_{\dot  W_{p}^{s, \frac{s}2}}, \quad
\| f|_{t=0} \|_{ \dot B_{p,q}^{ s -\frac{2}p } (\R )} \leq c \| f \|_{\dot  B_{p,q}^{s, \frac{s}2}}.
\end{align*}
\end{itemize}
\end{prop}

The following H$\ddot{\rm o}$lder type inequality is a direct consequence of Lemma 2.2 in \cite{chae}.

\begin{lemm}
\label{L25}
If $0 < s < \infty$ and $1 \leq p \leq r_i, p_i \leq \infty$ with $\frac1{r_i} + \frac1{p_i} = \frac1p$, $i=1,2$, then
\begin{align*}
\| f_1f_2\|_{ \dot B^{s, \frac{s}2}_{p,q}   }  
\lesssim \| f_1\|_{   \dot B^{s, \frac{s}2}_{p_1,q }   } \| f_2\|_{  L^{r_1}  }   + \| f_1\|_{ L^{p_2 } } \|f_2\|_{   \dot B^{s, \frac{s}2}_{r_2 ,q} }.
\end{align*}
\end{lemm}
%Note that \eqref{E28} holds for general homogeneous Besov space in
%$\R$ (see \cite{chae}).
%Let $f_i \in \dot B^{s, \frac{s}2}_{p,q} ({\mathbb R}^{3}  \times (0, \infty) ), \, i = 1,2$.
%By definition of $\dot B^{s, \frac{s}2}_{p,q} ({\mathbb R}^{3}  \times %(0, \infty) )$, there are $\tilde f_i$  be extensions of $f_i$ such %that $\|\tilde f_i \|_{\dot B^{\be}_{p} ( {\mathbb R}^{4}) } \leq c %\| f_i\|_{\dot B^{s, \frac{s}2}_{p,q} ({\mathbb R}^{3}  \times (0, %\infty) ) }$.

\begin{proof}
Let $E$ be the extension operator.
From  Lemma 2.2 in \cite{chae}, we get
\begin{align*}%\label{1019-3}
\notag \| f_1f_2\|_{\dot B^{s, \frac{s}2}_{p,q}  } & \leq  \| E( f_1) E( f_2 )\|_{  \dot B^{s, \frac{s}2}_{p,q} ( {\mathbb R}^{4}   ) }\\
\notag 
&\lesssim \| E f_1\|_{   \dot B^{s, \frac{s}2}_{p_1 ,q} ({\mathbb R}^{4}  ) } \| E f_2\|_{  L^{r_1} ({\mathbb R}^{4}  )}   + \|E f_1\|_{ L^{p_2 }({\mathbb R}^{4}  )} \| E f_2\|_{   \dot B^{s, \frac{s}2}_{r_2, q}({\mathbb R}^{4}   )} \\
&\lesssim \|   f_1\|_{   \dot B^{s, \frac{s}2}_{p_1,q }   } \|   f_2\|_{  L^{r_1}  }   + \|   f_1\|_{ L^{p_2 } } \|   f_2\|_{   \dot B^{s, \frac{s}2}_{r_2,q } }.
\end{align*}
\end{proof}

%Similarly, we have
%\begin{align*}
%\| f_1f_2\|_{  B^{\be,\frac{\be}2}_{p} ( \R \times (0, \infty)) } & \leq
%        c \big( \|   f_1\|_{     B^{\be,\frac{\be}2}_{s_1 } (\R \times (0, \infty)  ) } \|   f_2\|_{  L^{r_1} ((\R \times (0, \infty)   )}   + \|   f_1\|_{ L^{s_2 }((\R \times (0, \infty) )} \|   f_2\|_{     B^{\be,\frac{\be}2}_{r_2 }((\R \times (0, \infty)    )}  \big).
%\end{align*}

\begin{lemm}
\label{L26}
Let $ 1 \leq  q \leq p  \leq   \infty$.
\begin{itemize}
\item[(1)]
If $1 \le p < \infty$ and $1 \le r \le \infty$, then $\dot B^{\frac{5}q,\frac{5}{2q }}_{q,r}   \subset L^{p,r} $ with
\begin{align*}
\| f \|_{L^{p,r}  } \lesssim \| f \|_{\dot B^{\frac{5}q -\frac{5}p,\frac{5}{2q }- \frac{5}{2p  }}_{q,r}  },
\end{align*}
where $L^{p,r} $ is Lorentz space on $\R  \times (0, \infty )$. \\
In particular, if $p =r$, then
\begin{align*}
\| f \|_{L^{p}  } \lesssim \| f \|_{\dot B^{\frac{5}q -\frac{5}p,\frac{5}{2q }- \frac{5}{2p }}_{q,p}  }.
\end{align*}
\item[(2)]
If $p = \infty$, then $\dot B^{\frac{5}q,\frac{5}{2q }}_{q, 1}   \subset L^\infty $ with
\begin{align*}
\| f \|_{L^\infty  } \lesssim \| f \|_{\dot B^{\frac{5}q,\frac{5}{2q }}_{q,1}  }.
\end{align*}
\end{itemize}
\end{lemm}

\begin{proof}
See Appendix \ref{AA}.
\end{proof}

From Proposion \ref{P24} ,  Lemma \ref{L25} and Lemma \ref{L26}, we obtain the following lemma.

\begin{lemm}
\label{L27}
Let $ 0 < \be\leq \ga < \infty$,  $1 \leq  p_1, \, q\leq \infty$ and  $p_1<   p_2, \, p_3 \leq \infty$ such that $ \frac1{p_1} = \frac1{p_2} + \frac1{p_3}$
and   $\be -\frac{5}{p_2} = \ga - \frac{5}{p_1}$. 
If $ p_1 < r_1 < \infty$, then 
\begin{align}
\label{E28}
\begin{split}
\| f g\|_{\dot B^{\be, \frac{\be}2}_{p_1, q}   }  
&\lesssim \| f \|_{L^{p_3} }\|  g \|_{ \dot B^{\be, \frac{\be}2}_{p_2, q }   } + \| f \|_{ \dot B^{\be, \frac{\be}2}_{p_2, q }   } \|  g \|_{ L^{p_3}   }\\
&\lesssim \| f \|_{\dot B^{\frac{5}{p_1} -\frac{5}{p_3}, \frac{5}{2p_1} -\frac{5}{2p_3}}_{p_1, p_3} }  \| g \|_{\dot B^{\ga,\frac{\ga}2}_{p_1, q}  } +  \| f\|_{\dot B^{\ga,\frac{\ga}2}_{p_1, q}  }  \| g\|_{\dot B^{\frac{5}{p_1} -\frac{5}{p_3}, \frac{5}{2p_1} -\frac{5}{2p_3}}_{p_1, p_3} }.
\end{split}
\end{align}
If $ p_3 =\infty $, then
\begin{align}
\label{E29}
\begin{split}
\| f g\|_{\dot B^{\be, \frac{\be}2}_{p_1, q}   }  
&\lesssim \| f \|_{L^{\infty } }\|  g \|_{ \dot B^{\be, \frac{\be}2}_{p_1, q }   } + \| f \|_{ \dot B^{\be, \frac{\be}2}_{p_1, q }   } \|  g \|_{ L^{\infty}   }\\
&\lesssim \| f \|_{\dot B^{\frac{5}{p_1}, \frac{5}{2p_1}}_{p_1, 1} }  \| g \|_{\dot B^{\be,\frac{\be}2}_{p_1, q}  } +  \| g\|_{\dot B^{\frac{5}{p_1}, \frac{5}{2 p_1} }_{p_1, 1} }  \| f\|_{\dot B^{\be,\frac{\be}2}_{p_1, q}  }.
 \end{split}
\end{align}
\end{lemm}

\begin{lemm}
\label{L28}
If $ 1 < p < \infty$, $1 \leq q \leq \infty$, and $ s > \frac2p$, then 
\begin{align*}
\dot W^{ s,\frac{s}2}_{p}  \subset  C_b ([0, \infty); \dot  B^{s -\frac{2}p}_{p,p} (\R)), \quad
\dot B^{s, \frac{s}2}_{p,q}  \subset  C_b ([0, \infty); \dot  B^{s -\frac{2}p}_{p,q} (\R)).
\end{align*}
\end{lemm}

\begin{proof}
See Appendix \ref{AB}.
\end{proof}

\begin{lemm}
\label{L29}
If $k \in \N_0$ and $\be \in \N_0^3 $ such that  $ s > k + |\be|$, then for $1 \leq p, \, q \leq \infty$
\begin{equation}
\label{E210}
\| D_t^k D_x^{ |\be|} f\|_{\dot B^{s -|\be| -2k, \frac12 (s -\be -2k)}_{p,q}  ( {\mathbb R}^{4} )} \lesssim \| f\|_{\dot B^{ s, \frac{s}2 }_{p,q} ( {\mathbb R}^{4}) }
\end{equation}
and for $1 < p < \infty$
\begin{equation}
\label{E211}
\| D^k_t D_x^{|\be|} f\|_{\dot W^{s - \be -2k,\frac12 ( s -\be -2k)}_{p} ({\mathbb R}^{4})} \lesssim \| f\|_{\dot W^{s, \frac{s}2 }_{p} ({\mathbb R}^{4})}.
\end{equation}
\end{lemm}

\begin{proof}
See Appendix \ref{AC}.
\end{proof}

%\begin{defin}[Weak solution to the
%] Let $1<p <\infty$ and $\frac{5}p \leq \al$.
%Let $u_0$  satisfy the same hypothesis as in Theorem \ref{thm-navier}.
%A vector field $b\in B^{\al,\frac{\al}2}_{p}(\R \times (0, \infty))$and  $u\in B^{\al-1,\frac{\al-1}2}_{p}(\R \times (0, \infty))$ is called a weak solution of the non-Newtonian fluids \eqref{maineq2} if the following variational formulations are satisfied:\\
% \begin{align}\label{weaksolution-NS}
%  -\int^\infty_0\int_{\R} \Big( u\cdot \Phi_t +  \big(   +(u\otimes u) \big):\nabla \Phi \Big) dxdt = \int_{\Rn} u_0 (x) \cdot \Phi(x,0) dx
%\end{align}
%for each $\Phi\in C^\infty_0(\R\times [0,\infty))$ with $\mbox\divg  _x\Phi=0$
%. In addition, for each $\Psi\in C^1_c(\R)$, $u$ satisfies  \eqref{Stokes-bvp-2200}.
%\end{defin}
%

The fundamental solution of the  heat equation is
\begin{align*}
\Gamma(x,t)  = \left\{\begin{array}{ll} \vspace{2mm}
\displaystyle\frac{1}{\sqrt{4\pi t}^3} e^{ -\frac{|x|^2}{4t} }, &   t > 0,\\
0, & t < 0.
\end{array}
\right.
\end{align*}

Let
\[
\Ga* f(x,t) = \int_0^t \int_{\R} \Ga (x-y, t-s) f (y,s) dyds
\]
and 
\[
\Ga g(x,t) = \int_{\R} \Ga (x-y, t) g (y) dy.
\]

\begin{prop}
[\cite{Triebel2}]
\label{P31}
For $1 \le p \le \infty$ and $k \ge 0$, 
\[
\| \Ga g\|_{\dot W^{2k, k}_p } \lesssim \| g \|_{\dot B^{2k -\frac2p}_{p,p} (\R)}.
\]
For $1 \le p,q \le \infty$ and $s > 0$, 
\[
\|\Ga g\|_{\dot B^{s, \frac{s}2}_{p,q}  } \lesssim \| g \|_{\dot B^{s -\frac2p}_{p,q} (\R)}. 
\]
\end{prop}

\begin{prop}
\label{P32}
Let $1< p < \infty$ and $1 \leq q \leq \infty$. 
\begin{itemize}
\item[(1)]
Let  $f\in \dot B^{s_1,\frac{s_1}2}_{p_1,q} $ for some
$(s_1, p_1)$ satisfying the conditions $p_1\leq p$, $0<s_1<s <1$, $\frac5{p_1}-\frac5{p} + s-s_1=1$. Then 
\begin{align*}
%\label{E12}
\| D_x \Ga* f\|_{ \dot B^{s,\frac{s}2 }_{p,q} }
 &\lesssim \|f\|_{\dot B^{s_1,\frac{s_1}2}_{p_1,q} }.
\end{align*}

\item[(2)]
If $1 < s $, then 
\[
\|  D_x  \Ga* f \|_{ \dot B^{s,\frac{s} 2 }_{p,q}({\mathbb R}^3 \times (0,\infty))} \lesssim \|f\|_{\dot B^{s-1 ,\frac{s } 2 -\frac12}_{p,q}({\mathbb R}^3 \times (0,\infty))}.
\]
\end{itemize}
\end{prop}

\begin{proof}
See Appendix \ref{AD}.
\end{proof}

\section{Proof of  Theorem \ref{T1}}
\label{S3}
\setcounter{equation}{0}

\subsection{Approximating solutions}

Let $F^0 = G^0 = H^0 =0$.
Let$(u^1,p^1)$ and $b^1$ be the solution of the equations
\begin{align*}
\left\{\begin{array}{l}\vspace{2mm}
u^1_t - \De u^1 + \na p^1 =0 \quad \mbox{ in } \quad \R\times (0,\infty), \\ 
\vspace{2mm}
b^1_t -\De  b^1 = 0 \quad \mbox{ in } \quad \R\times (0,\infty), \\
\vspace{2mm}
{\rm div} \, u^1 = {\rm div} \, b^1 = 0 \quad \mbox{ in } \quad \R\times (0,\infty),\\
u^1|_{t=0} = u_0, \quad b^1|_{t=0} = b_0.
\end{array}
\right.
\end{align*}

If we have $(u^{m}, p^{m})$ and $b^m$ for some $m \in \N_0$, then let us $(u^{m+1}, p^{m+1})$ and $b^{m+1}$ be the solution of the equations
\begin{equation}
\label{E31}
\begin{split}
\left\{\begin{array}{l} \vspace{2mm}
u^{m+1}_t - \De u^{m+1} + \na p^{m+1} =  -\nabla\cdot \,  ( u^m \otimes u^m ) + (\na \times b^m) \times b^m  \quad \mbox{in} \quad \R\times (0,\infty),\\
\vspace{2mm}
\vspace{2mm} {\rm div}\,  u^{m+1} = 0 \quad \mbox{in} \quad \R\times (0,\infty), \\
\vspace{2mm}
b^{m+1}_t -\De  b^{m+1} =  - \na \times ((\nabla \times b^m) \times b^m) + \nabla \times (u^m \times b^m) \quad \mbox{in} \quad \R\times (0,\infty),\\
 u^{m+1}|_{t=0}= u_0, \quad  b^{m+1}|_{t =0} = b_0.
\end{array}
\right.
\end{split}
\end{equation}

Note that $u^{m+1}$ and $b^{m+1}$ can be written in the form
\begin{align*}
%\label{E12}
%\begin{split}
u^{m+1}(x,t)  & = \int_{\R} \Ga(x-y, t) u_0 (y) dy +  \int_0^t \int_{\R} \na \Ga (x-y, t-s) {\mathbb P} (u^m \otimes u^m ) (y,s)dyds\\
& \qquad  - \int_0^t \int_{\R} \na \Ga(x-y, t-s)   :  {\mathbb P}\big(  ( b^m \otimes b^m) \big)(y,s) dyds\\
& \qquad   +  \frac12  \int_0^t \int_{\R}    {\mathbb P} \na \Ga(x-y, t-s)  | b^m (y,s)|^2 dyds,\\
%\end{split}\\
%\label{E13}
%\begin{split}
b^{m+1}(x,t) &  = \int_{\R} \Ga(x -y, t) b_0 (y) dy +  \int_0^t \int_{\R}   \na^2 \Ga(x-y,t-s)b^m (y,s) \times b^m (y,s)   dyds\\
& \qquad  +  \int_0^t \int_{\R } \na \Ga(x-y, t-s) \times ( u^m\times b^m )(y,s) dyds,
%\end{split}
\end{align*} 
Let 
\begin{align*}
M^u_{\al, q} : =  \|u_0\|_{    \dot  B^{\al-\frac{2}{p}}_{p,q}({\mathbb
R}^{n})} ,\quad
M^b_{\al,q} : =  \|b_0\|_{   \dot   B^{\al-\frac{2}{p}}_{p,q}({\mathbb
R}^{n})} .
\end{align*}
Since $\frac{3}p -1 > -\frac2p$, from  Proposition \ref{P31}, we have
\begin{equation}
\label{E32}
\begin{split}
&\| u^1\|_{   \dot B^{\al -1,\frac{\al -1}{2} }_{p,q} }  \leq c M^u_{\al -1, q},
\quad \   \| u^1\|_{   \dot B^{\frac{5}p -1,\frac{5}{2p} -\frac12  }_{p,q} } \leq c M^u_{\frac{5}p -1,q} \\
& \| b^1\|_{   \dot B^{\frac{5}p -1,\frac{5}{2p} -\frac12  }_{p,q} }  \leq c M^b_{\frac{5}p-1,q},
\quad  \| b^1\|_{   \dot B^{\al,\frac{\al }2}_{p,q}  }    \leq c M^b_{\al, q}.
\end{split}
\end{equation}

Let $1 < \al < 2$.
Let $\frac{2 -\al}{5} <\de < \frac1{5}$ and
\begin{align*}
\frac1{p_1} = \frac1p +\de, \quad \be = 5\de +\al -2,\quad 
  \frac{5}{p_2} =\frac{5}p +\be -\al+1. %\quad  \frac{5}{p_3} =\frac{5}p +\be -\al+1 .
\end{align*}
Note that
\begin{align}
\label{E33}
 0 < \be < \al-1 <   1, \quad \frac{5}{p_1} -\frac{5}p +\al-1 -\be =1
\end{align}
and
\begin{align}
\label{E34}
\frac1{p_2} + \frac1{5} = \frac1{p_1}, \quad \be -\frac{5}{p_2} = \al -1 -\frac{5}p.
%\label{1027-5}\frac1{p_3} + \frac2{5} = \frac1{p_1}, \quad \be -\frac{5}{p_3} = \al -1 -\frac{5}p
\end{align}

Since $ 0< \al -1 < 1$, from  \eqref{E33}, Proposition \ref{P31}    and Proposition \ref{P32},   we have
\begin{align}
\label{E35}
\|  u^{m+1} \|_{  \dot B^{\al-1,\frac{\al -1}2}_{p,q}  } 
&\lesssim M^u_{\al -1, q} +      \sum_{1 \leq i, j \leq 3} \|u^m_i u^m_j\|_{\dot B^{\be, \frac{\be}2}_{p_1, q} }  +   \sum_{1 \leq i, j \leq 3} \|  b^m_i   b^m_j \|_{ \dot B^{\be, \frac{\be}2}_{p_1, q}  }.
\end{align}

From Lemma \ref{L27}, we have
\begin{equation}
\label{E36}
\begin{split}
\| u^m_i u^m_j\|_{\dot B^{\be, \frac{\be}2}_{p_1, q}   }
&\lesssim \| u^m \|_{\dot B^{\frac{5}p -1, \frac{5}{2p} -\frac12}_{p, 5} }  \| u^m  \|_{\dot B^{\al -1,\frac{\al-1}2}_{p, q}  }  \quad 1 \leq i, j \leq 3, \\
\| b^m_i b^m_j \|_{\dot B^{\be, \frac{\be}2}_{p_1, q}   }  
&\lesssim \| b^m \|_{\dot B^{\frac{5}p -1, \frac{5}{2p} -\frac12}_{p, 5} } \| b^m  \|_{\dot B^{\al -1,\frac{\al-1}2}_{p, q}  } \quad 1 \leq i, j \leq 3.
\end{split}
\end{equation}
%From from \eqref{E34} and (2) of  Proposition \ref{P24}, we have
%\begin{align}\label{1027-10}
%\begin{split}
%\|  u^m  \|_{  \dot B^{\be,\frac{\be}2}_{p_2, q} }
%\leq  c\| u^m  \|_{\dot B^{\al -1,\frac{\al-1}2}_{p, q}  },\quad
%\|   b^m  \|_{  \dot B^{\be,\frac{\be}2}_{p_2, q} },
%\leq  c\| b^m  \|_{\dot B^{\al -1,\frac{\al-1}2}_{p, q}  }
%\end{split}
%\end{align}
%and from Lemma \ref{L26}, then   we have
%\begin{align}\label{1027-6}
%\begin{split}
%\| u^m\|_{L^{5}}
 %\leq  c \| u^m \|_{\dot B^{\frac{5}p -1, \frac{5}{2p} -\frac12}_{p, 5} },\quad
%\leq  c \| u^m \|_{ B^{\al -1, \frac{\al}{2} -\frac12}_{p} },\\
 %\| b^m\|_{L^{5}}
%\leq  c \| b^m \|_{\dot B^{\frac{5}p-1, \frac{5}{2p} -\frac12}_{p, 5} }.% \leq  c   \| b^m \|_{ B^{\al-1, \frac{\al}{2}-\frac12}_{p} }.
%\end{split}
%\end{align}
Hence, from \eqref{E34}-\eqref{E36},  for $1 < \al <2$, we have
\begin{align}
\label{E37}
\begin{split}
\| u^{m+1}\|_{     \dot B^{\al -1,\frac{\al -1}2}_{p,q} } 
&\lesssim M_{\al -1,q}^u +   \| u^m \|_{\dot B^{\frac{5}p -1, \frac{5}{2p} -\frac12}_{p, 5} }\|   u^m  \|_{ \dot  B^{\al -1,\frac{\al-1}2}_{p,q} } 
+ \| b^m \|_{\dot B^{\frac{5}p-1, \frac{5}{2p} -\frac12}_{p, 5} } \|   b^m  \|_{   \dot  B^{\al -1,\frac{\al-1}2}_{p,q} }.
\end{split}
\end{align}
If $\al > 2$, then from Proposition \ref{P32}, we have
\[
\|  u^{m+1} \|_{  \dot B^{\al-1,\frac{\al -1}2}_{p,q}  }
\lesssim M_{\al -1,q}^u +       \sum_{1 \leq i, j \leq 3} \|u^m_i  u^m_j\|_{\dot B^{\al-2, \frac{\al -2}2}_{p, q} }  +   \sum_{1 \leq i, j \leq 3}\|  b^m_i   b^m_j\|_{ \dot B^{\al -2, \frac{\al -2}2}_{p, q}  }.
\]
From Lemma \ref{L27}, we have
\begin{align}
\label{E38}
\begin{split}
\| b^m_i b^m_j \|_{\dot B^{\al-2,\frac{\al-2}2 }_{p,q} }
%   &  \leq  c  \| b^m   \|_{ \dot B^{\al -2,\frac{\al-2 }2 }_{p_3, q}   }\|    b^m \|_{L^{5} }\\
&\lesssim \|     b^m \|_{\dot B^{\al -1,\frac{\al -1}2}_{p,q} } \| b^m \|_{\dot B^{\frac{5}p -1, \frac{5}{2p} -\frac12}_{p,5} } \quad 1 \leq i, j \leq 3.
\end{split}
\end{align}
Thus, we obtain \eqref{E37} for $ \al > 2$.

Next, we estimate $b^{m+1}$.  If $ 1 < \al < 2$, then from Proposition \ref{P31} and  Proposition \ref{P32}, we have
\begin{equation}
\label{E39}
\|b^{m+1} \|_{  \dot B^{\al-1,\frac{\al -1}2}_{p,q}  } 
\lesssim M_{\al -1,q}^b +       \|u^m_i  b^m_j\|_{\dot B^{\be, \frac{\be}2}_{p_1, q} }  +   \|  b^m_i   u^m_j\|_{ \dot B^{\be, \frac{\be}2}_{p_1, q}  } 
+ \|b^m_i   b^m_j \|_{    \dot B^{\al-1,\frac{\al }2-\frac12}_{p,q}}.
\end{equation}
By Lemma \ref{L27}, we get
\begin{align}
\label{E310}
\begin{split}
 \|    b^m_i  b^m_j \|_{    \dot B^{\al-1,\frac{\al }2-\frac12}_{p,q} }  
&\lesssim \|     b^m \|_{\dot B^{\al -1,\frac{\al -1}2}_{p,q} } \| b^m \|_{\dot B^{\frac{5}p, \frac{5}{2p}}_{p,1} } \,\, 1 \leq i, j \leq 3, \\
 \| u^m_i  b^m_j\|_{\dot B^{\be, \frac{\be}2}_{p_1,q}} 
&\lesssim \| u^m \|_{\dot B^{\frac{5}p -1, \frac{5}{2p} -\frac12}_{p, 5} }\|   b^m  \|_{ \dot  B^{\al -1,\frac{\al-1}2}_{p,q} } +
\| b^m \|_{\dot B^{\frac{5}p-1, \frac{5}{2p} -\frac12}_{p, 5} } \|   u^m  \|_{   \dot  B^{\al -1,\frac{\al-1}2}_{p,q} } \, \,  1 \leq i, j \leq 3.
 \end{split}
\end{align}
According to \eqref{E39} and \eqref{E310}, we have
\begin{align}
\label{E311}
\begin{split}
  \|  b^{m+1} \|_{  \dot B^{\al-1,\frac{\al -1}2}_{p,q}  }
&\lesssim M_{\al-1, q}^b+   \| u^m \|_{\dot B^{\frac{5}p -1, \frac{5}{2p} -\frac12}_{p, 5} }\|   b^m  \|_{ \dot  B^{\al -1,\frac{\al-1}2}_{p,q} }\\
& \qquad +
\| b^m \|_{\dot B^{\frac{5}p-1, \frac{5}{2p} -\frac12}_{p, 5} } \|   u^m  \|_{   \dot  B^{\al -1,\frac{\al-1}2}_{p,q} } 
+ \|     b^m \|_{\dot B^{\al -1,\frac{\al -1}2}_{p,q} } \| b^m \|_{\dot B^{\frac{5}p, \frac{5}{2p}}_{p,1} }.
\end{split}
\end{align}
If $  2 < \al $, then by Proposition \ref{P31} and Proposition \ref{P32},  we have
\begin{align*}
  \|  b^{m+1} \|_{  \dot B^{\al-1,\frac{\al -1}2}_{p,q}  }
&\lesssim M^b_{\al -1,q} +  \sum_{1 \leq i, j \leq 3}\|u^m_i   b^m_j\|_{\dot B^{\al -2, \frac{\al -2}2}_{p_1, q} }   + \sum_{1 \leq i, j \leq 3}\|b^m_i  b^m_j\|_{    \dot B^{\al-1,\frac{\al }2-\frac12}_{p,q} }.
\end{align*}
 Applying \eqref{E38} and \eqref{E310}, we obtain  \eqref{E311} for $2 < \al$.
 
For  $ 1< \al  $, then from  Proposition \ref{P31},  Proposition \ref{P32}, and \eqref{E310},  we have
\begin{align}
\begin{split}
\|  b^{m+1} \|_{{\dot B^{\al, \frac{\al}2}_{p,q} }  } 
& \leq c \big( M_{\al,q}^b +       \sum_{1 \leq i, j \leq 3}\|u^m_i   b^m_j\|_{    \dot B^{\al -1,\frac{\al }2 -\frac12}_{p,q} }  +   \sum_{1 \leq i, j \leq 3}\|    b^m_i  b^m_j \|_{    \dot B^{\al,\frac{\al }2}_{p,q} } \big)\\
& \leq c \big( M_{\al,q}^b +  \| b^m \|_{\dot B^{\frac{5}p-1, \frac{5}{2p} -\frac12}_{p, 5} } \| u^m \|_{ \dot B^{\al-1, \al-\frac12 }_{p,q} } \\
&  \qquad     +
        \| u^m \|_{\dot B^{\frac{5}p-1, \frac{5}{2p}-\frac12}_{p,5} }\| b^m \|_{ \dot B^{\al, \frac{\al}{2}}_{p,q}    }      + \| b^m \|_{\dot B^{\frac{5}p, \frac{5}{2p}}_{p,1} }\|     b^m \|_{\dot B^{\al,\frac{\al}2}_{p,q} }\big).
        \end{split}
\end{align}
For the case $\al =\frac{5}p$, from \eqref{E37} and \eqref{E311},   we have
\begin{align}
\label{E314}
\begin{split}
  \|  u^{m+1} \|_{  \dot B^{\frac{5}p-1, \frac{5}{2p} -\frac12}_{p, 5}  }
   &  \leq c \big(  M_{\frac{5}p-1,5}^u +   \| u^m \|^2_{\dot B^{\frac{5}p -1, \frac{5}{2p} -\frac12}_{p,5} } +    \| b^m \|^2_{\dot B^{\frac{5}p-1, \frac{5}{2p} -\frac12}_{p,5} } \big),
\end{split}
\end{align}
\begin{align}
\label{E315}
\notag \|  b^{m+1} \|_{{\dot B^{\frac{5}p -1, \frac{5}{2p} -\frac12 }_{p,5} }  } & \leq c \big( M_{\frac{5}p -1,5} + \| u^m \|_{\dot B^{\frac{5}p -1, \frac{5}{2p} -\frac12}_{p, 5} }\| b^m \|_{ \dot B^{\frac{5}p-1, \frac{5}{2p}-\frac12}_{p, 5}    } \\
& \quad    +  \| b^m \|_{\dot B^{\frac{5}p, \frac{5}{2p}}_{p,1} } \| b^m \|_{ \dot B^{\frac{5}p-1, \frac{5}{2p}-\frac12 }_{p, 5 } }
          \big),
\end{align}
\begin{align}
\begin{split}
\label{E316}
\|  b^{m+1} \|_{  \dot B^{\frac{5}p,\frac{5}{2p}}_{p,1}   } & \leq c \big( M_{\frac{5}p,1} +  \| b^m \|_{ \dot B^{\frac{5}p-1, \frac{5}{2p}-\frac12 }_{p,5}} \| u^m \|_{ \dot B^{\frac{5}p, \frac{5}{2p}}_{p ,1}  } \\
&  \qquad     +
        \| u^m \|_{\dot B^{\frac{5}p-1, \frac{5}{2p}-\frac12}_{p,5} }\| b^m \|_{ \dot B^{\frac{5}p, \frac{5}{2p}}_{p,1}    }      + \| b^m \|^2_{\dot B^{\frac{5}p, \frac{5}{2p}}_{p,1} }\big).
\end{split}
\end{align}
Take $\ep < 1$ and  $M > 0$ satisfying
\begin{align*}
M_{\frac{5}p-1,5}^u, \,\, M_{\frac{5}p-1,5}^b, \,\, M_{\frac{5}p,1}^b < \frac{\ep}{8c},\quad 
  M_{\al-1, q}^u, \,\,  M_{\al-1,q}^b, \,\,  M^b_{\al,q} <  \frac1{4c} M.
\end{align*}
Then, from \eqref{E32}, we have
\begin{align*}
  \|  u^{1} \|_{\dot B^{\frac{5}p-1, \frac{5}{2p}-\frac12 }_{p,5} } < \ep, \quad     \|  b^{1} \|_{\dot B^{\frac{5}p-1, \frac{5}{2p}-\frac12 }_{p,5} } < \ep, \quad \|  b^{1} \|_{\dot B^{\frac{5}p, \frac{5}{2p} }_{p,1} } < \ep,\\
  \|  u^{1} \|_{\dot B^{\al-1, \frac{\al-1}{2}}_{p,q} }   < \frac12 M < M, \quad
 \|  b^{1} \|_{\dot B^{\al-1, \frac{\al-1}{2}}_{p,q} }   < \frac12 M < M,\quad  \|  b^{1} \|_{\dot B^{\al, \frac{\al}{2}}_{p,q} }   < \frac12 M < M.
\end{align*}
Suppose that
\begin{align*}
   \|  u^{m} \|_{\dot B^{\frac{5}p-1, \frac{5}{2p}-\frac12 }_{p,5} } < \ep, \quad     \|  b^{m} \|_{\dot B^{\frac{5}p-1, \frac{5}{2p}-\frac12 }_{p,5} } < \ep, \quad \|  b^{m} \|_{\dot B^{\frac{5}p, \frac{5}{2p} }_{p,1} } < \ep,\\
  \|  u^{m} \|_{\dot B^{\al-1, \frac{\al-1}{2}}_{p,q} }    < M, \quad
 \|  b^{m} \|_{\dot B^{\al-1, \frac{\al-1}{2}}_{p,q} }    < M,\quad  \|  b^{m} \|_{\dot B^{\al, \frac{\al}{2}}_{p,q} }   < M.
\end{align*}
Then, from  \eqref{E37},  \eqref{E314},  \eqref{E315} and \eqref{E316}, we have
\begin{align*}
\|  u^{m +1} \|_{\dot B^{\frac{5}p-1,  \frac{5}{2p}-\frac12 }_{p, 5} }   &\leq c \big( \frac12 \ep + \ep^2\big) < \ep, \\
\|  b^{m +1} \|_{\dot B^{\frac{5}p-1,  \frac{5}{2p}-\frac12 }_{p, 5} }
%& \leq \big(\frac13 \ep +  \frac13\ep +  \ep^{1 +\te_2} M^{ 1 -\te_2}  \big)\\
&\leq \big(\frac13 \ep + \frac13 \ep +  \frac13\ep   \big) = \ep \\
\|  b^{m+1} \|_{  \dot B^{\frac{5}p,\frac{5}{2p}}_{p,1}   } 
&\leq \big(\frac13 \ep + \frac13 \ep +  \frac13\ep   \big) =\ep \\
\| u^{m+1}\|_{     \dot B^{\al -1,\frac{\al -1}2}_{p,q } } 
&\leq c \big( M_{\al -1,q}^u + \ep M\big)    <    M \\
 \|  b^{m+1} \|_{  \dot B^{\al,\frac{\al}2}_{p}   }
&\leq  \big(\frac12 M + \ep M \big) < M.
\end{align*}
Hence for all $m \in \N$,
\begin{equation}
\label{E317}
\| u^m\|_{\dot B^{\frac{5}p-1, \frac{5}{2p} -\frac12}_{p,5 }}, \quad \| b^m \|_{\dot B^{\frac{5}p-1, \frac{5}{2p}-\frac12}_{p,5} }, \,\,  \|b^m \|_{\dot B^{\frac{5}p, \frac{5}{2p}}_{p,1} } <\ep, \quad 
\|  u^{m} \|_{\dot B^{\al-1, \frac{\al-1}{2}}_{p,q} }, \,\,   \|  b^{m} \|_{\dot B^{\al, \frac{\al}{2}}_{p,q} }   < M.
\end{equation}

\subsection{Uniform convergence}

Let $U^m=u^{m}-u^{m-1}$, $B^m = b^{m} -b^{m-1}$ and $P^m=p^{m}-p^{m-1}$.
Then $(U^m,B^m,P^m)$ satisfies the equations
\begin{align*}
U^{m +1}_t - \De U^{m +1} + \na P^{m +1} &  =- (U^{m}\cdot \na u^{m} +u^{m-1}\cdot  \na U^{m}) - ( B^{m} \cdot \na b^{m} +  b^{m-1} \cdot \na B^{m-1} ),\\
  \hspace{30mm} \divg U^{m +1}& =0, \\
B^{m +1}_t - \De B^{m +1}  &  =- (u^m\cdot \na B^{m-1}+U^{m-1}\cdot \na b^{m-1})
 + (b^m\cdot \na U^{m-1} + B^{m-1}\cdot \na u^{m-1})\\
 & \qquad \  -\na \times \big( ( \na \times  B^{m-1} ) \times b^m - ( \na \times  b^{m-1} ) \times B^{m-1} \big)
\end{align*}
in $\R\times (0,\infty)$ with initial conditions $U^m|_{t=0} = 0$ and $B^m|_{t=0} = 0$.
According to the previous calculations, we have 
\begin{align*}
  \|  U^{m+1} \|_{  \dot B^{\frac{5}p-1, \frac{5}{2p} -\frac12}_{p, 5}  }&  \leq c \big(  \| u^{m-1} \|_{\dot B^{\frac{5}p -1, \frac{5}{2p} -\frac12}_{p,5} } +    \| b^{m-1} \|_{\dot B^{\frac{5}p-1, \frac{5}{2p} -\frac12}_{p,5} }  +  \| u^m \|_{\dot B^{\frac{5}p -1, \frac{5}{2p} -\frac12}_{p,5} } +    \| b^m \|_{\dot B^{\frac{5}p-1, \frac{5}{2p} -\frac12}_{p,5} } \big)\\
  & \qquad \big(    \| U^m \|_{\dot B^{\frac{5}p -1, \frac{5}{2p} -\frac12}_{p,5} } +    \| B^m \|_{\dot B^{\frac{5}p-1, \frac{5}{2p} -\frac12}_{p,5} } \big)\\
& < 4\ep \big(    \| U^m \|_{\dot B^{\frac{5}p -1, \frac{5}{2p} -\frac12}_{p,5} } +    \| B^m \|_{\dot B^{\frac{5}p-1, \frac{5}{2p} -\frac12}_{p,5} } \big)\\
& < \frac14 \big(    \| U^m \|_{\dot B^{\frac{5}p -1, \frac{5}{2p} -\frac12}_{p,5} } +    \| B^m \|_{\dot B^{\frac{5}p-1, \frac{5}{2p} -\frac12}_{p,5} } \big),
\end{align*}
\begin{align*}
\notag \|  B^{m+1} \|_{{\dot B^{\frac{5}p -1, \frac{5}{2p} -\frac12 }_{p,5} }  } & \leq  \big(  \| u^{m-1} \|_{\dot B^{\frac{5}p -1, \frac{5}{2p} -\frac12}_{p,5} } +    \| b^{m-1} \|_{\dot B^{\frac{5}p-1, \frac{5}{2p} -\frac12}_{p,5} }  +  \| u^m \|_{\dot B^{\frac{5}p -1, \frac{5}{2p} -\frac12}_{p, 5} } \\
& \quad +    \| b^m \|_{\dot B^{\frac{5}p-1, \frac{5}{2p} -\frac12}_{p,5} }
 + \| b^{m-1} \|_{ \dot B^{\frac{5}p, \frac{5}{2p} }_{p, 1}    } + \| b^{m} \|_{ \dot B^{\frac{5}p, \frac{5}{2p} }_{p,1}    }  \big)\\
  & \qquad \big(    \| U^m \|_{\dot B^{\frac{5}p -1, \frac{5}{2p} -\frac12}_{p, 5} } +    \| B^m \|_{\dot B^{\frac{5}p-1, \frac{5}{2p} -\frac12}_{p,5} }  +    \| B^m \|_{\dot B^{\frac{5}p, \frac{5}{2p} }_{p,1} } \big)\\
& \leq 6\ep\big(    \| U^m \|_{\dot B^{\frac{5}p -1, \frac{5}{2p} -\frac12}_{p, 5} } +    \| B^m \|_{\dot B^{\frac{5}p-1, \frac{5}{2p} -\frac12}_{p,5} }   +    \| B^m \|_{\dot B^{\frac{5}p, \frac{5}{2p} }_{p,1} }\big)\\
& < \frac14  \big(    \| U^m \|_{\dot B^{\frac{5}p -1, \frac{5}{2p} -\frac12}_{p, 5} } +    \| B^m \|_{\dot B^{\frac{5}p-1, \frac{5}{2p} -\frac12}_{p,5} }  +     \| B^m \|_{\dot B^{\frac{5}p, \frac{5}{2p} }_{p,1} } \big),
\end{align*}
and 
\begin{align*}
\notag \|  B^{m+1} \|_{{\dot B^{\frac{5}p, \frac{5}{2p} }_{p,1} }  } & \leq  \big(  \| u^{m-1} \|_{\dot B^{\frac{5}p -1, \frac{5}{2p} -\frac12}_{p,5} } +    \| b^{m-1} \|_{\dot B^{\frac{5}p-1, \frac{5}{2p} -\frac12}_{p,5} }  +  \| u^m \|_{\dot B^{\frac{5}p -1, \frac{5}{2p} -\frac12}_{p, 5} } \\
& \quad +    \| b^m \|_{\dot B^{\frac{5}p-1, \frac{5}{2p} -\frac12}_{p,5} }
 + \| b^{m-1} \|_{ \dot B^{\frac{5}p, \frac{5}{2p} }_{p, 1}    } + \| b^{m} \|_{ \dot B^{\frac{5}p, \frac{5}{2p} }_{p,1}    }  \big)\\
  & \qquad \big(    \| U^m \|_{\dot B^{\frac{5}p -1, \frac{5}{2p} -\frac12}_{p, 5} } +    \| B^m \|_{\dot B^{\frac{5}p-1, \frac{5}{2p} -\frac12}_{p,5} }  +    \| B^m \|_{\dot B^{\frac{5}p, \frac{5}{2p} }_{p,1} } \big)\\
& \leq 6\ep\big(    \| U^m \|_{\dot B^{\frac{5}p -1, \frac{5}{2p} -\frac12}_{p, 5} } +    \| B^m \|_{\dot B^{\frac{5}p-1, \frac{5}{2p} -\frac12}_{p,5} }   +    \| B^m \|_{\dot B^{\frac{5}p, \frac{5}{2p} }_{p,1} }\big)\\
& < \frac14  \big(    \| U^m \|_{\dot B^{\frac{5}p -1, \frac{5}{2p} -\frac12}_{p, 5} } +    \| B^m \|_{\dot B^{\frac{5}p-1, \frac{5}{2p} -\frac12}_{p,5} }  +     \| B^m \|_{\dot B^{\frac{5}p, \frac{5}{2p} }_{p,1} } \big).
\end{align*}
Therefore, 
\begin{equation}
\label{E318}
\begin{split}
&\|  B^{m+1} \|_{{\dot B^{\frac{5}p -1, \frac{5}{2p} -\frac12 }_{p, 5} }  } + \|  U^{m+1} \|_{{\dot B^{\frac{5}p -1, \frac{5}{2p} -\frac12 }_{p, 5} }  } + \|  B^{m+1} \|_{{\dot B^{\frac{5}p, \frac{5}{2p} }_{p,1} }  }\\
& \quad  < \frac12 \big(\|  U^{m} \|_{{\dot B^{\frac{5}p -1, \frac{5}{2p} -\frac12 }_{p, 5} }  }   + \|  B^{m} \|_{{\dot B^{\frac{5}p -1, \frac{5}{2p} -\frac12 }_{p, 5} }  } + \|  B^{m} \|_{{\dot B^{\frac{5}p, \frac{5}{2p} }_{p,1} }  } \big).
\end{split}
\end{equation}
This implies that $(u^m,b^m)$ is Cauchy sequence in $ \dot B^{\frac{5}p-1, \frac{5}{2p} -\frac12}_{p, 5} \times \big( \dot B^{\frac{5}p-1, \frac{5}{2p} -\frac12}_{p, 5}  \cap \dot B^{\frac{5}p, \frac{5}{2p} }_{p, 1}  \big)$. Hence,  $ u^m$  and $b^m$ converge to $u$ and $b$ in $\dot B^{\frac{5}p-1, \frac{5}{2p} -\frac12}_{p, 5}$ and $\dot B^{\frac{5}p-1, \frac{5}{2p} -\frac12}_{p, 5}  \cap \dot B^{\frac{5}p, \frac{5}{2p}}_{p, 1}$, respectively.

\subsection{Existence}

Let $u$ and $b$ be the solution constructed in the previous section. Since $u^m$ and $b^m$ converge to $u$ and $b$ in $\dot B^{\frac{5}p-1, \frac{5}{2p} -\frac12}_{p, 5}$ and $\dot B^{\frac{5}p-1, \frac{5}{2p} -\frac12}_{p, 5}  \cap \dot B^{\frac{5}p, \frac{5}{2p}}_{p, 1}$, respectively, from \eqref{E317}, we have
\begin{align*}
\| u\|_{\dot B^{\frac{5}p-1, \frac{5}{2p} -\frac12}_{p,5}}, \,\, \| b \|_{\dot B^{\frac{5}p-1, \frac{5}{2p}-\frac12}_{p,5} } , \,\, \| b \|_{\dot B^{\frac{5}p, \frac{5}{2p}}_{p,1} }\leq \ep,\quad 
\|  u \|_{\dot B^{\al-1, \frac{\al-1}{2}}_{p,q} }, \,\,   \|  b \|_{\dot B^{\al, \frac{\al}{2}}_{p,q} }   \leq  M.
\end{align*}
In this section, we will show that $u$ and $b$ satisfy the weak formulation of Hall-MHD equations; that is, $u,b$ is a weak solution of Hall-MHD equations with the appropriate distribution $p$.
Let $\Phi\in C^\infty_{0}( \R \times  (0,\infty))$ with $\mbox{div }\Phi=0$.
Observe that
\begin{align*}
-\int^\infty_0\int_{\R} u^{m+1}\cdot \Delta\Phi dxdt&=\int^\infty_0\int_{\R}u^{m+1}\cdot \Phi_t+(u^m\otimes u^m): \nabla \Phi  - \big( b^m  \otimes b^m \big)  : \na \Phi dxdt\\
&\quad+<u_0,\Phi(\cdot,0)>,\\
-\int^\infty_0\int_{\R} b^{m+1}\cdot \Delta\Phi dxdt&=\int^\infty_0\int_{\R}b^{m+1}\cdot \Phi_t +  \big( u^m\times b^m \big)  \cdot \na  \times   \Phi +  (b^m \otimes b^m) :     \na  (\na \times     \Phi  )  dxdt\\
&\quad+<b_0,\Phi(\cdot,0)>.
\end{align*}
Since $ \dot B^{\frac{5}p-1, \frac{5}{2p} -\frac12}_{p, 5}  \subset L^{5} $ by Lemma \ref{L26}, $ u^m$ and $b^m$  converge to $u$ and $b$ in $L^{5}$, respectively,  and so do $ u^m \otimes u^m$,   $u^m \times b^m$ and $  b^m\otimes b^m$ to $u \otimes u$,  $u \times b$ and $  b \otimes b$ in $L^{\frac{5}2}$, respectively.
Now sending $m$ to $\infty$, we obtain
\begin{align*}
-\int^\infty_0\int_{\R} u\cdot \Delta\Phi dxdt&=\int^\infty_0\int_{\R}u\cdot \Phi_t+(u\otimes u): \nabla \Phi  - \big( b  \otimes b \big)  : \na \Phi dxdt\\
&\quad+<u_0,\Phi(\cdot,0)>,\\
-\int^\infty_0\int_{\R} b\cdot \Delta\Phi dxdt&=\int^\infty_0\int_{\R}b\cdot \Phi_t+  \big( u\times b \big)  \cdot \na  \times   \Phi + (b \otimes b) :  \na  (\na \times     \Phi  )   dxdt\\
&\quad+<b_0,\Phi(\cdot,0)>.
\end{align*}
Hence, $(u,b)$ is solution of \eqref{E11}.

\subsection{Uniqueness }
Let $ u_1\in    \dot B^{\frac{5}p -1, \frac{5}{2p}-\frac12}_{p, 5} $ and $ b_1\in  \dot B^{\frac{5}p-1, \frac{5}{2p}-\frac12}_{p,5}    \cap  \dot B^{\frac{5}p, \frac{5}{2p}}_{p, 1} $ be  another weak solution of \eqref{E11}  with pressure $p_1$ satisfying
\[
\| u_1\|_{\dot B^{\frac{5}p-1, \frac{5}{2p} -\frac12}_{p, 5}}, \,\, \| b_1 \|_{\dot B^{\frac{5}p-1, \frac{5}{2p}-\frac12}_{p, 5} } , \,\, \| b_1 \|_{\dot B^{\frac{5}p, \frac{5}{2p}}_{p, 1} }\leq \ep.
\]
Let $U = u - u_1,~B = b - b_1$ and $P = p -p_1$. 
Then $(U, B, P)$ satisfies the equations
\begin{align*}
U_t - \De U + \na P &  =- \big( (U\cdot \na ) u_{1}- (u\cdot  \na) U \big) - \big(( B \cdot \na) b_{1} -  (b \cdot \na ) B \big),\\
  \hspace{30mm} \divg U & =0, \\
B_t - \De B  &  =- \big((u_1\cdot \na ) B+ (U\cdot \na) b \big)
 + \big( (b_1\cdot \na)  U +  (B\cdot \na ) u \big)\\
 & \qquad \  -\na \times \big( ( \na \times  B ) \times b - ( \na \times  b_1 ) \times B \big)
\end{align*}
in $\R\times (0,\infty)$ with initial conditions $U|_{t=0} = 0$ and $B|_{t=0} = 0$.
The estimate  \eqref{E318} implies that
\[
\|  U \|_{ \dot B^{\frac{5}p-1, \frac{5}{2p} - \frac12}_{p, 5} }
     + \|  B \|_{ \dot B^{\frac{5}p-1, \frac{5}{2p} - \frac12}_{p, 5} } +  \|  B \|_{ \dot B^{\frac{5}p, \frac{5}{2p} }_{p,1} }\\
< \frac12  \big(\|  U \|_{ \dot B^{\frac{5}p-1, \frac{5}{2p} - \frac12}_{p, 5} }
     + \|  B \|_{ \dot B^{\frac{5}p-1, \frac{5}{2p} - \frac12}_{p, 5} } + \|  B \|_{ \dot B^{\frac{5}p, \frac{5}{2p} }_{p,1} }   \big).
\]
Hence $u \equiv u_1$ and $b =b_1$ in $\R \times (0, \infty)$.

This completes the proof of Theorem \ref{T1}.

\appendix
\setcounter{equation}{0}
\section{Proof of Lemma \ref{L26}}
\label{AA}
\setcounter{equation}{0}
%\label{proofofprposition3.2}
Let $\Phi = \phi_{-1} + \phi_0  + \phi_1$ and $\Phi_j(x,t) = 2^{5j} \Phi(2^j x, 2^{2 j} t)$ so that $\widehat{\Phi_j}(\xi, \tau) = \widehat{\Phi}(2^{-j} \xi, 2^{-2 j} \tau)$. Since $\phi_j = \phi_j * \Phi_j$, we have $ f* \phi_j = f* \Phi_j * \phi_j$. By Young's convolution inequality, for $\frac1p = \frac1r + \frac1{q} -1$, we have
\[
\|f* \phi_j \|_{L^p } \leq \| \Phi_j \|_{L^r}\|f* \phi_j \|_{L^{q} }
\leq c 2^{5j (\frac1{q} -\frac1p)}\|f* \phi_j \|_{L^{q} }.
\]
Hence, for $p = \infty$
\begin{align*}
\| f \|_{L^{\infty}  } & \leq  \sum_{j \in {\mathbb Z}}\|f* \phi_j \|_{L^{\infty} } \leq c \sum_{j \in {\mathbb Z}} 2^{\frac{5j}{q}} \|f* \phi_j \|_{L^{q} }
 = \| f \|_{\dot B^{\frac{5}q,\frac{5}{2 q }}_{q1}  }.
\end{align*}
This proves (2) of Lemma \ref{L26}.

We take $ 0 < \te < 1$ and $  q  \leq p < p_2$
satisfying $2 -\frac{5}{q} =  - \frac{5}{p_2}$ and  $ \frac1p = \frac{\te}{q} + \frac{1 -\te}{p_2}$. Note that $2 ( 1-\te) = \frac{5}q -\frac{5}p$.  From (2) of Proposition \ref{P21}, we have
\begin{align*}
\| f \|_{L^{p_2}  } \lesssim \| f \|_{\dot W^{2, 1}_{q}  }.
\end{align*}
Since $( L^{p_2} ,L^q )_{\te, r} = L^{p,r}  $ and
$( \dot W^{2, 1}_{q}, L^q )_{\te, r} = \dot B^{ 2(1 -\te), 1 -\te }_{q,r}  $, we have
\begin{align*}
\| f\|_{L^{p,r}  } \lesssim \| f\|_{\dot B^{ 1 -\te, \frac{1 -\te}{2}}_{q,r} } = \| f\|_{\dot B^{\frac{5}q -\frac{5}p, \frac{5}{2q} -\frac{5}{2p}}_{q,r}  }.
\end{align*}
This proves (1) of Lemma  \ref{L26}.

\section{Proof Lemma \ref{L28}}
\label{AB}
\setcounter{equation}{0}

Since the proofs are similar, we only prove the case $\dot B^{\al, \frac{\al}2}_{p,q}$.

From (4) in Proposition \ref{P21}, 
\begin{align*}
\dot B^{\al, \frac{\al}2}_{p,q}  \subset L^\infty (0, \infty; \dot  B^{\al -\frac2p}_{p,q} (\R)).
\end{align*}
Let $ \frac2p< \al <2$ and $f \in \dot B^{\al, \frac{\al}2}_{p,q} $ with $ f(x,0) =0$ for $x \in \R$. Let $\tilde f(x,t) = f(x,t)$ for $ t > 0 $ and $\tilde f(x,t) =0$ for $ t < 0$. 
Then $\tilde f \in \dot B^{\al, \frac{\al}2}_{p,q} ({\mathbb R}^{4}) $ with $ \| \tilde f\|_{\dot B^{\al, \frac{\al}2}_{p,q} (\R \times (-\infty, t) )} \lesssim \| f\|_{\dot B^{\al, \frac{\al}2}_{p,q} (\R \times (0, t))}$. 
From (4) of Proposition \ref{P21}, 
\begin{align*}
\| f(t) \|_{\dot B^{\al-\frac2p}_{p,q} (\R)} &\lesssim \| \tilde f\|_{ \dot B^{\al, \frac{\al}2}_{p,q} (\R \times (-\infty, t) )  } \lesssim \|  f\|_{ \dot B^{\al, \frac{\al}2}_{p,q} (\R \times (0, t))  } \ri 0 \quad \mbox{as} \quad t \ri 0.
\end{align*}

Let $F(x,t) = \Ga_t * f_0(x)$ for $f_0 (x) = f(x,0)$. Since  $f_0  \in \dot B^{\al -\frac2p}_{p,q} (\R)$, we have
$F \in \dot B^{\al, \frac{\al}2}_{p,q}$ with $\| F \|_{\dot B^{\al, \frac{\al}2}_{p,q} } \lesssim \| f_0 \|_{\dot B^{\al -\frac2p}_{p,q} (\R)}$ and $\| F(t) - f_0 \|_{\dot B^{\al -\frac2p}_{p,q} (\R)} \ri 0 $ as $t \ri 0$. From the above argument, we have $\| F(t) - f (t) \|_{\dot B^{\al -\frac2p}_{p,q} (\R)} \ri 0 $ as $t \ri 0 $. 
Thus,
\begin{align*}
\| f(t) - f_0\|_{\dot B^{\al -\frac2p}_{p,q} (\R)}  \leq
\| f(t) - F(t)\|_{\dot B^{\al -\frac2p}_{p,q} (\R)}  + \| F(t) - f_0\|_{\dot B^{\al -\frac2p}_{p,q} (\R)}  \ri 0 \quad \mbox{as} \quad t\ri 0.
\end{align*}
Therefore, $\dot B^{\al, \frac{\al}2}_{p,q}  \subset C ([0, \infty); \dot  B^{\al -\frac2p}_{p,q} (\R))$.
This proves Lemma \ref{L28} for $ \frac2p < \al < 2 $.

Let $2k +\frac2p< \al < 2k +2$ for $k \in \N $. 
Then we have $D_x^{2k} f \in \dot B^{\al -2k, \frac{\al}2 -k}_{p,q}$ and $D_x^{2k} f_0 \in \dot B^{\al -2k -\frac2p}_{p,q} (\R)$. 
By the same argument, we have
\begin{align*}
\| f(t) - f_0 \|_{\dot B_{p,q}^{\al -\frac2p} (\R)} \leq \| D_x^{2k} ( f(t) -   f_0 ) \|_{\dot B_{p,q}^{\al-2k -\frac2p} (\R)} \ri 0 \quad \mbox{as} \quad  t \ri 0.
\end{align*}
This proves Lemma \ref{L28} for $2k +\frac2p< \al < 2k +2$ for $k \in \N $.
For the case $ 2k   \leq \al \leq 2k  +\frac2p, \,\, k \in {\mathbb N}$, we use the  property of real  interpolation. We complete the proof of Lemma \ref{L28}.

%We fix $ 0 \leq s$. Let  $f_s(x,t) = f(x,t) - f(x,s)$ for $t \geq s$. Then, $ f_s \in \dot B^{\al, \frac{\al}2}_{p,q} (\R \times (s,\infty))$ with $f_s (x,s) =0$. Applying the above argument, we have

\section{Proof of Lemma \ref{L29}}
\label{AC}
\setcounter{equation}{0}

%Let $f \in \dot B^{\al s, s}_{p,q}  $.
Let $\Phi = \phi_{-1} + \phi_0  + \phi_1$ and $\Phi_j(x,t) = 2^{5} \Phi(2^j x, 2^{2 j} t)$ so that $\widehat{\Phi_j}(\xi, \tau) = \widehat \Phi(2^{-j} \xi, 2^{-2j} \tau)$. Since $\phi_j = \phi_j * \Phi_j$, we have
\begin{align*}
 \phi_j *(D_t^k D_x^{ \be} f) =  \rho_j * \phi_j * f,
\end{align*}
where   $  \widehat{\rho_j} (\xi, \tau) = \widehat \Phi(2^{-j} \xi, 2^{- 2j} \tau) (2\pi i \tau)^k  (-2 \pi \xi)^\be $. 
The $L^p(\mathbb{R}^{4})$-multiplier norm $M_j$ of $\widehat \rho_{j}(\xi, \tau)$ equals the $L^p(\mathbb{R}^{4})$-multiplier norm of
$\widehat {\rho_{j}^{'} } (\xi,\tau)  :=2^{(|\be| + 2k)j} \widehat  \Phi(\xi,\tau) (2\pi i\tau)^k  (-2\pi \xi )^{\be }$  for $ 1 \leq p \leq \infty$ (see
Lemma 6.1.5 in \cite{BL}). From Theorem 6.1.3 in \cite{BL}, the $L^p(\mathbb{R}^{4})$-multiplier norm of
$\widehat{\rho_{j}^{'} }$  is bounded by $2^{(|\be| +2 k)j}$. 
Thus, 
\begin{align*}
\| D^k_t  D_x^{\be} f\|_{\dot B^{s - |\be| -2k, \frac{s -|\be| -2k}2}_{p,q} ({\mathbb R}^{4})} &  = \big( \sum_{j \in {\mathbb Z}} 2^{ (s - |\be| -2k)qj}\| \phi_j * (D_t^k D_x^{\be} f)\|^q_{L^p ({\mathbb R}^{4})}\big)^\frac1q\\
&\lesssim \big( \sum_{j \in {\mathbb Z}} 2^{2 s  qj}  \| \phi_j * f\|^q_{L^p ({\mathbb R}^{4})}\big)^\frac1q\\
&= \| f\|_{\dot B^{s, \frac{s}2 }_{p,q} ({\mathbb R}^{4})}.
\end{align*}
This proves \eqref{E210}.

Note that $\frac{  |\xi|^{2|\be|} (2\pi i \tau)^k  }{ ( |\xi|^2 + 2\pi i \tau)^{|\be| + k}  }   $ is  $L^p ({\mathbb R}^{4})$-multiplier for $ 1 < p < \infty$ (see Theorem 4.6\'{} in \cite{St}).
Since
\begin{align*}
  ( |\xi|^2 + 2\pi i \tau)^{s -|\be| -k}   (-2\pi \xi)^{2\be} (2\pi i \tau)^k \widehat f(\xi, \tau)
  =  \frac{  |\xi|^{2\be} (2\pi i \tau)^k  }{ ( |\xi|^2 + 2\pi i \tau)^{|\be| + k}  }  ( |\xi|^2 + 2\pi i \tau)^{s }   \widehat f(\xi ,\tau) ,
\end{align*}
we have
\[
\| D^k_t D_x^\be f\|_{\dot W^{s - |\be| -2k, \frac{ s -\be -2k}2}_{p} ({\mathbb R}^{4})} 
\lesssim \|   f\|_{\dot W^{ s,\frac{ s}2 }_{p} ({\mathbb R}^{4})}.
\]
This proves \eqref{E211}. % the secon equation of  Lemma \ref{L29}.

\section{Proof of Proposition \ref{P32}}
\label{AD}
\setcounter{equation}{0}

Let $f \in \dot W^{s, \frac{s}2 }_{p} ({\mathbb R}^{4})$ for $s \in {\mathbb R}$. 
Let
\[
\Ga* f(x,t) = 
\begin{cases} \vspace{2mm}
\int_{-\infty}^t \int_{\R} \Ga (x-y, t-s) f(y,s) dyds & \quad t \geq 0,\\
<\Ga (x-\cdot, t -\cdot), f> & \quad  t < 0.
\end{cases}
\]
Here,  $<, >$ is duality pairing between $\dot W^{s, \frac{s}2 }_{p} ({\mathbb R}^{4})$ and $\dot W^{-s, -\frac{s}2 }_{p'} ({\mathbb R}^{4}) = (\dot W^{s, \frac{s}2 }_{p} ({\mathbb R}^{4}))'$  (the dual space of  $\dot W^{s, \frac{s}2 }_{p} ({\mathbb R}^{4})$), where $\frac1p + \frac1{p'} =1$.

Note that  $ {\mathcal F}_{x,t} (h_{s} * \Ga * f)(\xi, \tau) =  (|\xi|^2+2 \pi i\tau)^{\frac{s}2 -1 } {\mathcal F}_{x,t} (f)(\xi, \tau)$, $ s \in {\mathbb R}$, where $h_s$ is defined in Section \ref{S2}.
From the definition of $\dot W^{s, \frac{s}2 }_p ({\mathbb R}^{4} ) $, we have for $s \in {\mathbb R}$,
\[
\|\Ga* f\|_{\dot W^{s, \frac{s}2 }_p ({\mathbb R}^{4}) } = \|  f \|_{\dot W^{ s-2 , \frac{s}2 -1 }_p ({\mathbb R}^{4})}.
\]
Since $D_x \Ga * f = \Ga * D_x f$, we have 
\begin{align}
\label{AD41}
\|  D_x \Ga* f\|_{\dot W^{s, \frac{s}2 }_p ({\mathbb R}^{4} ) } &  = \| D_x f \|_{\dot W^{s-2, \frac{s}2-1 }_p ({\mathbb R}^{4})} 
\lesssim \| f \|_{\dot W^{s-1, \frac{s}2-\frac{1}2 }_p ({\mathbb R}^{4})}.
\end{align}

Let $  \dot W^{-1, -\frac12}_{p0}  $ be a  dual space of  $  \dot W^{1,  \frac12}_{p'} $. Let $ f \in  \dot W^{-1, -\frac12}_{p0} $ and  $ \tilde f \in \dot W^{-1, -\frac12}_{p}  $ be a zero extension of $f$, that is, $< \tilde f, \phi>_{{\mathbb R}^{4}} = <  f, \phi|_{\R \times (0, \infty)}>_{\R \times (0, \infty )}$, so that
$ \|\tilde f\|_{\dot W^{-1, -\frac12 }_{p} ({\mathbb R}^{4})} \leq c \| f \|_{ \dot W^{-1, -\frac12}_{p0} }$. Note that $\Ga  * \tilde f(x,t) = \Ga * f(x,t)$ for $(x,t) \in {\mathbb R}^{4}$. 
From \eqref{AD41} and (2) of Proposition \ref{P21}, we have
\begin{equation}
\label{AD42}
\begin{split}
\| D_x \int_0^t  \Ga (t -s) f(x,s) ds\|_{L^p  } 
= \|D_x  \Ga* \tilde f\|_{L^p ({\mathbb R}^{4})} 
\lesssim \|\tilde  f\|_{\dot W^{-1, -\frac12}_{p} ({\mathbb R}^{4})} 
\lesssim \| f\|_{\dot W^{-1, -\frac12 }_{p0}  }.
\end{split}
\end{equation}
Similarly, we obtain
\begin{align}
\label{AD43}
\|D_x \int_0^t  \Ga (t -s) f(x,s) ds\|_{\dot W^{1, \frac12}_p  } 
\lesssim \| f\|_{L^p}.
\end{align}
Using \eqref{AD42}, \eqref{AD43}, and the properties of   real interpolation of dual spaces, we have for $0<s<1$,
\[
\|  D_x \int_0^t  \Ga (t -s) f(x,s) ds\|_{\dot B^{s, \frac{s}2 }_{p,q}  }  \lesssim \| f\|_{\dot B^{ s -1, \frac{s}2 -\frac12 }_{p,q0}},
\]
where $\dot B^{s-1, \frac{s}2 -\frac12}_{p,q0}  $ is the dual space of  $  \dot B^{-s + 1, -\frac{s}2 + \frac12}_{p', q'} $.

Let $-1 -\frac{5}p = -\frac{5}r $.
Since $ \| f\|_{\dot W^{ -1, -\frac12}_{p0} } \leq c\| f\|_{L^r } $, from \eqref{AD42}, we have
\begin{align}
\label{AD44}
\|D_x \int_0^t  \Ga_\al (t -s) f(x,s) ds\|_{L^p } \lesssim \| f\|_{L^r }.
\end{align}

Let $\frac1{p^*} = \frac{1-s}r + \frac{s}p   $ and $-\frac{5}{p^*} = s_1 -\frac{5}{p_1}$. Using the real interpolation in \eqref{AD43} and \eqref{AD44}, and by Lemma \ref{L26}, we have for $0 < s < 1$,
\[
\|D_x \int_0^t  \Ga(t -s) f(x,s) ds\|_{\dot B^{s, \frac{s}2}_{p,q} } 
\lesssim \| f\|_{L^{p^*,q}} 
\lesssim \| f\|_{\dot B^{s_1,\frac{ s_1}2}_{p_1,  q}}.
\]
This proves (1) of Proposition \ref{P32}.

Let  $|\be| = 2k$ with $k  \in \N $.
Since $D_x^\be  D_x \int_0^t  \Ga (t -s) f(x,s) ds  =  D_x \int_0^t  \Ga (t -s)  D_x^\be   f(x,s) ds$,  from \eqref{AD43}, we have
\begin{align}
\label{AD45}
\begin{split}
&\|D_x \int_0^t  \Ga (t -s) f(x,s) ds\|_{L^p (0, \infty; \dot W^{k}_p (\R))} \\
&\le \|  D_x \int_0^t  \Ga (t -s)D_x^{\be -1}  f(x,s) ds\|_{L^p(0, \infty; \dot W^1 (\R)  }\\
&\lesssim \| D_x^{\be -1}  f\|_{L^p}
\lesssim \| f\|_{L^p(0,\infty;  \dot W^{2k-1}_p (\R))} 
\lesssim \| f\|_{\dot W^{2k-1, k - \frac{1}2 }_p}.
\end{split}
\end{align}
%We used Proposition \ref{prop230112-1} for the second inequality.

Since
\begin{align*}
D_t^k  D_x \int_0^t  \Ga  (t -s) f(x,s) ds = (-1)^k \De^{k }  D_x \int_0^t  \Ga (t -s) f(x,s) ds + \sum_{l =0}^{k-1} (-1)^l D_t^{l} \De^{k -1-l} D_x f,
\end{align*}
we have, from \eqref{AD45} and Lemma \ref{L29}, 
\begin{equation}
\label{AD46}
\begin{split}
&\|  D_x \int_0^t  \Ga (t -s) f(x,s) ds\|_{L^p (\R; \dot W^k_p (0, \infty))}\\
&\lesssim \|  D_x \int_0^t  \Ga (t -s) f(x,s) ds\|_{L^p (0, \infty; \dot W^{2 k}_p (\R))} + \| \sum_{l =0}^{k-1}   D_t^{l} \De^{k -1-l} D_x f\|_{L^p } \\
&\lesssim \| f\|_{\dot W^{ 2k -1, k-\frac12 }_p }.
\end{split}
\end{equation}
From \eqref{AD45}, \eqref{AD46} and  (3) of Proposition \ref{P24}, we obtain that for $k \in \N$,
\begin{equation}
\label{AD47}
\|  D_x \int_0^t  \Ga (t -s) f(x,s) ds \|_{   \dot W^{2k,k}_p } 
\lesssim \| f\|_{\dot W^{2k-1, k -\frac12 }_p }.
\end{equation}
Using the real interpolation in \eqref{AD43} and \eqref{AD47} and the Sobolev imbedding (see (2) of Proposition \ref{P24}), we have for $1 \le s$, 
\[
\|  D_x \int_0^t  \Ga (t -s) f(x,s) ds\|_{\dot B^{s, \frac{s}2}_{p,q} }  
\lesssim \| f\|_{\dot B^{s -1, \frac{s}2 -\frac12}_{p,q} }.
\]
This proves (2) of Proposition \ref{P32}.

\end{document}